\documentclass[11pt,leqno,twoside]{article}
\usepackage[english]{babel}
\usepackage{mathrsfs,amssymb,amsfonts,amsthm,amsmath,natbib,amsbsy}
\usepackage{dsfont,verbatim,fancyhdr}
\usepackage{graphicx,float}
\usepackage[hang]{caption}
\usepackage[usenames,dvipsnames]{xcolor}
\usepackage[bookmarks,bookmarksnumbered,colorlinks=true,pdfstartview=FitV, linkcolor=blue,citecolor=blue!90!green,urlcolor=blue!90!green]{hyperref}
\usepackage{tikz}
\usepgflibrary{arrows}
\usepackage{tkz-graph}
\usepackage{tikz}
\usepgflibrary{arrows}
\usepackage[ruled,vlined]{algorithm2e}
\usepackage{graphicx,float,subfigure,booktabs}

\usepgflibrary{arrows}
\usepackage[normalem]{ulem}
\usepackage{marginnote,enumerate,enumitem}
\usepackage{pdfsync}

\allowdisplaybreaks
\def\set@page@layout#1#2{%
  \setlength\textwidth{#1}
  \@settopoint\textwidth
  \setlength\textheight{#2}
  \@settopoint\textheight
  \setlength\@tempdima        {\paperwidth}
  \addtolength\@tempdima      {-\textwidth}
  \setlength\oddsidemargin    {.5\@tempdima}
  \addtolength\oddsidemargin  {-1in}
  \setlength\evensidemargin   {\oddsidemargin}
  \@settopoint\oddsidemargin
  \@settopoint\evensidemargin
  \setlength\topmargin{\paperheight}
  \addtolength\topmargin{-2in}
  \addtolength\topmargin{-\headheight}
  \addtolength\topmargin{-\headsep}
  \addtolength\topmargin{-\textheight}
  \addtolength\topmargin{-\footskip}     
  \addtolength\topmargin{-.5\topmargin}
  \@settopoint\topmargin
}

\bibpunct{\textcolor{blue!90!green}{(}}{\textcolor{blue!90!green}{)}}{\textcolor{blue!90!green}{;}}{a}{\textcolor{blue!90!green}{,}}{\textcolor{blue!90!green}{;}}

\def\d{\mathrm{d}}
\def\X{\mathcal{X}}
\def \Y{\mathcal{Y}} 
\def\M{\mathcal{M}} 
\def\MM{\mathscr{M}} 
\def\Z{\mathbb{Z}_+}
\def\r{r} 

\def\N{\mathbb{N}}
\def \R{\mathbb{R}}
\def \E{\mathbb{E}}
\def\F{\mathcal{F}}
\def\L{\mathcal{L}} 
\def\Gwf{\mathcal{A}_{K}} 
\def\Gcir{\mathcal{B}} 
\def\Gn{B} 
\def\A{\mathcal{A}}

\def\Fdir{\mathscr{F}_{\Pi}}
\def\Fgamma{\mathscr{F}_{\Gamma}}

\def\DK{\Delta_{K}}
 
\newcommand{\bs}[1]{\boldsymbol{#1}}

\newcommand{\norm}[1]{|{#1}|}
\def\na{\theta}

\def \aa {\bs\alpha}
\def \yy {\mathbf{y}}

\def \ee {\mathbf{e}}

\def \mm {\mathbf{m}}
\def \nn {\mathbf{n}}
\def \ii {\mathbf{i}}
\def \oo {\mathbf{0}}

\def \xx {\mathbf{x}}
\def \zz {\mathbf{z}}
\def \XX {\mathbf{X}}

\def \ZZ {\mathbf{Z}}
\def \ss {s} 
\def \SS {S} 

\def\hw{\widehat{w}}

\newtheorem{theorem1}{Theorem}[section]
\newenvironment{theorem}{\begin{theorem1}}{\end{theorem1}}
\newtheorem{proposition}[theorem1]{Proposition}
\newtheorem{definition1}[theorem1]{Definition}

\newtheorem{remark1}[theorem1]{Remark}

\newtheorem{example1}[theorem1]{Example}

\newtheorem{lemma1}[theorem1]{Lemma}
\newenvironment{lemma}{\begin{lemma1}}{\end{lemma1}}

\usepackage{ifthen}
\newcommand{\todo}[1]{ \ifthenelse{ \equal{\draft}{true} }{{\color{red} #1 }}{}}
\newcommand{\draft}{true}

\long\def\symbolfootnote[#1]#2{\begingroup\def\thefootnote{\hspace*{-1mm}\fnsymbol{footnote}}\footnote[#1]{#2}\endgroup}

\fancypagestyle{plain}{
\fancyhead{}
}

\title{\bf \vspace{-2.5cm}
Conjugacy properties of time-evolving Dirichlet and gamma
random measures
}
\author{
\normalsize\textsc{Omiros Papaspiliopoulos}\\[1mm]
\normalsize\emph{ICREA and Universitat Pompeu Fabra}\\[2mm]
\normalsize\textsc{Matteo Ruggiero
}\\[1mm]
\normalsize\emph{University of Torino and Collegio Carlo Alberto}\\[2mm]
\normalsize\textsc{Dario Span\`o}\\[1mm]
\normalsize\emph{University of Warwick}
}
\date{\today}

\begin{document}

\maketitle
\thispagestyle{empty}

\vspace{-5mm}
\begin{center}
\begin{minipage}{.75\textwidth}
\footnotesize\noindent
We extend classic characterisations of posterior distributions under Dirichlet process and gamma random measures priors to a dynamic framework. We consider the problem of learning, from indirect observations, two families of time-dependent processes of interest in Bayesian nonparametrics: the first is a dependent Dirichlet process driven by a Fleming--Viot model, and the data are random samples from the process state at discrete times; the second is a collection of dependent gamma random measures driven by a Dawson--Watanabe  model, and the data are collected according to a Poisson point process with intensity given by the process state at discrete times. Both driving processes are diffusions taking values in the space of discrete measures whose support varies with time, and are stationary and reversible with respect to Dirichlet and gamma priors respectively.  A common methodology is developed to obtain in closed form the time-marginal posteriors given past and present data. These are shown to belong to classes of finite mixtures of Dirichlet processes and gamma random measures for the two models respectively, yielding conjugacy of these classes to the type of data we consider. We provide explicit results on the parameters of the mixture components and on the mixing weights, which are time-varying and drive the mixtures towards the respective priors in absence of further data. Explicit algorithms are provided to recursively compute the parameters of the mixtures. Our results are based on the projective properties of the signals and on certain duality properties of their projections.  \\[-2mm]


\end{minipage}
\end{center}

\linespread{1.2}


\newpage

\tableofcontents

\section{Introduction}


\subsection{Motivation and main contributions}

An active area of research in Bayesian nonparametrics is the
construction and the statistical learning of so-called dependent
processes. These aim at accommodating weaker forms of dependence than exchangeability among the data, such as partial exchangeability in the sense of de Finetti. The task is then to let the infinite-dimensional parameter, represented by a
random measure, depend on a covariate, so that the generated data are exchangeable only conditional on the same covariate value, but not overall exchangeable. This approach was  inspired by \cite{ME99,ME00} and
has received considerable attention since. 

In the context of this article, the most relevant strand of this literature attempts to build time evolution into standard random measures  for
semi-parametric time-series analysis, combining the merits of flexible exchangeable modelling afforded by
random measures with those of mainstream generalised linear and time
series modelling.
For the case of Dirichlet
processes, the reference model in Bayesian nonparametrics introduced by \cite{F73}, the time evolution has often been built into the process by
exploiting its celebrated stick-breaking representation \citep{S94}. For example, 
\cite{D06} models the dependent process as an autoregression
with Dirichlet distributed innovations, 
\cite{Cea06} models the noise in a dynamic linear model with a
Dirichlet process mixture, 
\cite{CDD07} develops a time-varying Dirichlet mixture with reweighing
and movement of atoms in the stick-breaking representation, 
\cite{RT08} induces the dependence in time only via the atoms in the
stick-breaking representation, by
making them into an heteroskedastic random walk. See also \cite{CT12,Cea16,GS06,Gea16,MR16}. The stick-breaking representation of the Dirichlet process has demonstrated its versatility for constructing dependent processes, but makes it hard to derive any analytical information on the posterior structure of the quantities involved. 
Parallel to these
developments, random measures have been combined with hidden Markov
time series models, either for allowing the size of the latent space to evolve in
time using transitions 
based on a hierarchy of Dirichlet processes, e.g. 
\cite{BGR02}, \cite{VSTG08},
\cite{SGGL09} and \cite{ZZZ14}, or for building flexible emission
distributions that link the latent states to data, e.g.
\cite{YPRH11}, \cite{GR16}.

From a probabilistic perspective, there is a fairly canonical way to build stationary processes with
marginal distributions specified as random measures using stochastic
differential equations. This more principled approach to building time
series with given marginals has been well explored, both
probabilistically and statistically, for finite-dimensional marginal
distributions, either using processes with discontinuous sample paths,
as in \cite{BNS} or \cite{G11}, or using diffusions, as we undertake here. The
relevance of measure-valued diffusions in Bayesian nonparametrics has been pioneered in \cite{Wea07}, whose
construction naturally allows for separate control of the marginal
distributions and the memory of the process. 

The statistical models we
investigate in this article, introduced in Section \ref{sec:hmmeas}, can be seen as instances of what we call
\emph{hidden Markov measures}, since the models are formulated as
hidden Markov models where the latent, unobserved signal is a measure-valued
infinite-dimensional Markov process.  The signal in the first model is
the  Fleming--Viot (FV) process, denoted 
$\{X_{t},t\ge0\}$ on some state space $\Y$ (also called type space in population genetics), which admits the law of a
Dirichlet process on $\Y$ as marginal distribution. At times $t_n$, conditionally on  $X_{t_n}=x$, observations are drawn independently from $x$, i.e.,
\begin{equation}\label{DP data}
Y_{t_{n},i}\mid x\overset{iid}{\sim}x, \quad i=1,\ldots,m_{t_{n}},\quad
Y_{t_{n},i} \in \Y\,.
\end{equation}
Hence, this statistical model is a dynamic extension of the classic
Bayesian nonparametric model for unknown distributions of  \cite{F73}
and \cite{A74}. The signal in the second model is  
the Dawson--Watanabe (DW) process, denoted $\{Z_{t},t\ge0\}$ and also defined on $\Y$, that admits the law of a gamma random measure as marginal distribution. At times $t_n$,
conditionally on $Z_{t_n}=z$, the observations are a Poisson process $Y_{t_{n}}$
on $\Y$ with random 
intensity $z$, i.e., for any collection of disjoint
sets $A_1,\ldots,A_K\in\Y$ and  $K\in\N$, 
\[
Y_{t_{n}}(A_i) | z \overset{ind}{\sim}\text{Po}(z(A_{i})). 
\] Hence, this is a time-evolving Cox
process and can be seen as a dynamic extension of the classic
Bayesian nonparametric model for Poisson point processes of
\cite{L82}. 

The Dirichlet and the gamma random measures, used as Bayesian nonparametric
priors, have conjugacy properties to observation models of the type
described above, which have been exploited both for developing theory
and for building simulation algorithms for posterior and
predictive inference. These
properties, reviewed in Sections \ref{subsec: FV-static} and \ref{subsec: DW-static}, have propelled the use of these models into mainstream
statistics, and have been used directly in simpler models or to build
updating equations within
Markov chain Monte Carlo and variational Bayes computational
algorithms in hierarchical models. 

In this article, for the first time, we show that the
dynamic versions of these Dirichlet and gamma models also enjoy certain conjugacy properties. First, we formulate such models as hidden Markov models
where the latent signal is a measure-valued diffusion and the
observations arise at discrete times according to the mechanisms
described above. We then obtain that the filtering distributions, that
is the laws of the signal at each observation time conditionally on all
data up to that time, are finite mixtures of Dirichlet and gamma random
measures respectively. We provide a concrete posterior characterisation of these marginal distributions and explicit 
algorithms for the recursive computation of the parameters of these
mixtures. Our results show that these families of finite mixtures are closed with respect to the Bayesian learning in this dynamic framework, and thus provide an extension of the classic
posterior characterisations of \cite{A74} and \cite{L82} to time-evolving settings. 

The techniques we use to establish the new conjugacy results are
detailed in Section \ref{sec: computable filtering}, and build upon three aspects: the characterisations of Dirichlet and gamma random measures through their projections; certain results on measure-valued diffusions related to their time-reversal; and some very recent developments in \cite{PR14} that relate optimal filtering for
finite-dimensional hidden Markov models with the notion of duality for
Markov processes, reviewed in Section \ref{sec: filtering and
  duality}. 
  \begin{figure}[t!]
\begin{center}
\begin{tikzpicture}[>=triangle 45,scale=.6]
\node (fv1) at (-2,3) [rectangle,draw,line width=.5pt]
	{$\L(X_{t_{n}}\mid Y_{1:n})$};
\node (fv2) at (12,3) [rectangle,draw,line width=.5pt]
	{$\L(X_{t_{n+k}}\mid Y_{1:n})$};
\node (wf1) at (-2,0) [rectangle,draw,line width=.5pt]
	{$\L(X_{t_{n}}(A_{1},\ldots,A_{K})\mid Y_{1:n})$};
\node (wf2) at (12,0) [rectangle,draw,line width=.5pt]
	{$\L(X_{t_{n+k}}(A_{1},\ldots,A_{K})\mid Y_{1:n})$};
\draw[->] (fv1) -- (fv2);
\node at (4.8,3.3) {\footnotesize\text{time propagation}};
\draw[->,dashed] (fv1) -- (wf1);
\node at (-0.7,1.5) {\footnotesize\text{projection}};
\draw[->] (wf1) -- (wf2);
\draw[<-,dashed] (fv2) -- (wf2);
\node at (10,1.5) {\footnotesize\text{characterisation}};
\node at (4.8,0.3) {\footnotesize\text{time propagation}};
\end{tikzpicture}
\caption{\scriptsize  Scheme of the general argument for obtaining
  the filtering distribution of hidden Markov models with
  FV and DW signals, proved in Theorems \ref{thm: FV propagation} and
  \ref{thm: DW propagation}. In this figure $X_t$ is the latent
  measure-valued signal. Given data $Y_{1:n}$, the future distribution of the signal $\L(X_{t_{n+k}}\mid Y_{1:n})$ at time $t_{n+k}$ is determined by taking its finite-dimensional projection
$\L(X_{t_{n}}(A_{1},\ldots,A_{K})\mid Y_{1:n})$ onto an arbitrary partition $(A_{1},\ldots,A_{K})$, evaluating the relative propagation $\L(X_{t_{n+k}}(A_{1},\ldots,A_{K})\mid Y_{1:n})$ at time $t_{n+k}$, and by exploiting the projective characterisation of the filtering distributions.}\label{fig: scheme}
\end{center}
\end{figure}
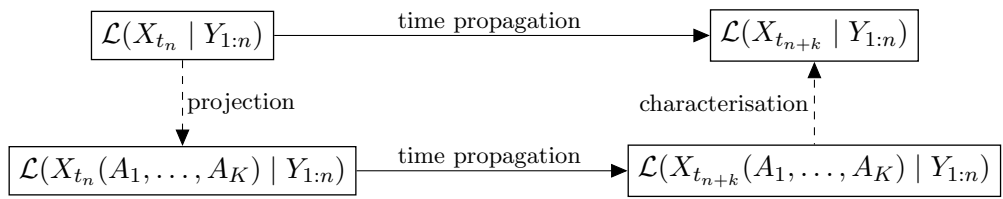
Figure \ref{fig: scheme} schematises, from a high level perspective, the strategy for obtaining our results. In a nutshell, the essence of our
theoretical results is that the operations of projection and
propagation of measures commute. More specifically, we first exploit the
characterisation of the Dirichlet and gamma random measures via their
finite-dimensional distributions, which are Dirichlet and independent
gamma distributions respectively. Then we exploit the fact that the 
dynamics of these finite-dimensional distributions induced by the
measure-valued signals are the Wright--Fisher (WF) diffusion and a multivariate Cox--Ingersoll--Ross (CIR) diffusion. Then, we extend the results in \cite{PR14}  to show that filtering these finite-dimensional signals on the basis of observations generated as described above results in
mixtures of Dirichlet and independent gamma distributions. Finally, we
use again the characterisations of Dirichlet and gamma measures via their finite-dimensional distributions to obtain the main results in this
paper, that the filtering distributions in the Fleming--Viot model
evolves in the family of finite mixtures of Dirichlet processes and those in the Dawson--Watanabe model in the family of finite mixtures of gamma random measures, under the observation models considered. The validity of this argument is formally proved in Theorems \ref{thm: FV
  propagation} and \ref{thm: DW propagation}. 
The resulting recursive procedures for  Fleming--Viot and
Dawson--Watanabe signals that describe how to compute the parameters
of the mixtures at each observation time  are given in Propositions \ref{prop: algorithm FV} and \ref{prop: algorithm DW},  and the associated pseudo codes are outlined in Algorithms \ref{alg: FV} and \ref{alg: DW}. 

The  paper is organised as follows. 
Section \ref{intro: hmms} briefly introduces some basic concepts on hidden Markov models. Section \ref{sec: illustration} provides a simple illustration of the underlying structures implied by previous results on filtering one-dimensional WF and CIR processes. These will be the reference examples throughout the paper and provide relevant intuition on our main results in terms of special cases, since the WF and CIR model are the one-dimensional projections of the infinite-dimensional families we consider here. 
Section \ref{sec:hmmeas} describes the two families of dependent  random measures which are the object of this contribution, the Fleming--Viot and the Dawson--Watanabe diffusions, from a non technical viewpoint. Connections of the dynamic models with the their marginal or static sub-cases given by Dirichlet and gamma random measures, well known in Bayesian nonparametrics, are emphasised. 
Section \ref{sec: main results} exposes and discusses the main results on the conjugacy properties of the two above families, given observation models as described earlier, together with the implied algorithms for recursive computation. All the technical details related to the strategy for proving the main results and to the duality structures associated to the signals are deferred to Section \ref{sec: computable filtering}.


\subsection{Hidden Markov models}\label{intro: hmms}

Since our time-dependent Bayesian nonparametric
models are formulated as hidden Markov models, we introduce here some basic related notions. 
A hidden Markov model (HMM) is a double sequence
$\{(X_{t_{n}},Y_{n}),n\ge0\}$ where $X_{t_{n}}$ is an unobserved Markov chain, called \emph{latent signal}, and
$Y_n:=Y_{t_{n}}$ are conditionally independent observations given the signal. Figure \ref{fig: HMM} provides a graphical representation of an HMM. We will assume here that the signal is the discrete time sampling of a continuous time Markov process $X_{t}$ with  transition kernel  $P_t(x,\d x')$. The  signal 
parametrises the law of the observations $\L(Y_n| X_{t_{n}})$,
called \emph{emission distribution}.  When this law admits density, this will be denoted by $f_{x}(y)$. 
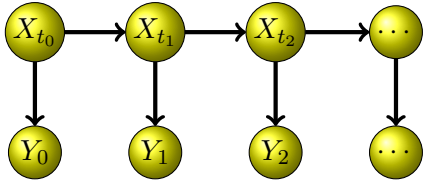
\begin{figure}
\begin{center}
\begin{tikzpicture}[scale=.4]
 \SetGraphUnit{2}
 \GraphInit[vstyle=Shade]
 \SetGraphShadeColor{yellow}{black}{yellow}
 \tikzset{LabelStyle/.style= {draw,
                              fill  = yellow,
                              text  = black}}
\tikzset{VertexStyle/.style = {shape = circle,
                                shading = ball,
                                ball color = yellow,
                                minimum size = 20pt,
                                inner sep = 1pt,
                                draw,
			       color=black}}
\SetGraphUnit{5}
\Vertex[L=\text{$X_{t_{0}}$},x=4*0,y=2*0]{1}
\Vertex[L=\text{$X_{t_{1}}$},x=4*1,y=2*0]{2}
\Vertex[L=\text{$X_{t_{2}}$},x=4*2,y=2*0]{3}
\Vertex[L=\text{$\cdots$},x=4*3,y=2*0]{4}
\Vertex[L={$Y_{0}$},x=4*0,y=-2*2]{5}
\Vertex[L={$Y_{1}$ },x=4*1,y=-2*2]{6}
\Vertex[L={$Y_{2}$},x=4*2,y=-2*2]{7}
\Vertex[L={$\cdots$ },x=4*3,y=-2*2]{8}
\tikzset{EdgeStyle/.style = {->, ultra thick}}
\Edge(1)(2)
\Edge(2)(3)
\Edge(3)(4)
\Edge(1)(5)
\Edge(2)(6)
\Edge(3)(7)
\Edge(4)(8)
\tikzset{VertexStyle/.style = {}}
 \end{tikzpicture}
\begin{minipage}{.75\textwidth}
\caption{\scriptsize  Hidden Markov model represented as a graphical model.}\label{fig: HMM}
\end{minipage}
\end{center}
\end{figure}

Filtering optimally an HMM requires the sequential exact evaluation of
the so-called \emph{filtering distributions} $\L(X_{t_{n}} | Y_{0:n})$, i.e., the laws of the signal at different times given past and present observations. Denote
$\nu_n:=\L(X_{t_n} | Y_{0:n})$ and let $\nu$ be
the prior distribution for $X_{t_0}$.
The exact or optimal filter is the solution of the recursion
\begin{equation}\label{recursion}
\nu_0 = \phi_{Y_{t_{0}}} (\nu)\,,\quad \nu_{n} = \phi_{Y_{t_{n}}}(\psi_{ t_n - t_{n-1}}(\nu_{n-1})),\quad\,n\in\N.
\end{equation}
This involves the following two operators acting on measures: the \emph{update operator}, which in case a density exists takes the form
\begin{equation}\label{update operator}
\phi_{y}(\nu)(\d x)
=\frac{f_x(y) \nu(\d x)}{p_\nu(y)}, \quad \quad p_\nu(y) = \int_{\X} f_x(y)
\nu(\d x)\,,
\end{equation}
and the \emph{prediction operator}
\begin{equation}\label{prediction operator}
\psi_t(\nu)(\d x')
=\int_{\X}\nu(\d x)P_t(x,\d x').
\end{equation}
The update operation amounts to an application of Bayes' Theorem to the currently available distribution conditional on the incoming data. The prediction operator propagates forward the current law of the signal of time $t$ according to the transition kernel of the underlying continuous-time latent process. The above recursion \eqref{recursion} then alternates update given the incoming data and prediction of the latent signal as follows: 
\begin{equation*}
\L(X_{t_{0}})
\overset{\text{update}}{\longrightarrow }
\L(X_{t_{0}}\mid Y_{0})
\overset{\text{prediction}}{\longrightarrow }
\L(X_{t_{1}}\mid Y_{0})
\overset{\text{update}}{\longrightarrow }
\L(X_{t_{1}}\mid Y_{0},Y_{1})
\overset{\text{prediction}}{\longrightarrow }\cdots
\end{equation*} 
If $X_{t_{0}}$ has prior $\nu=\L(X_{t_{0}})$, then $\nu_{0}=\L(X_{t_{0}} | Y_{0})$ is the posterior conditional on the data observed at time $t_{0}$; $\nu_{1}$ is the law of the signal at time $t_{1}$ obtained by propagating $\nu_{0}$ of a $t_{1}-t_{0}$ interval and conditioning on the data $Y_{0},Y_{1}$ observed at time $t_{0}$ and $t_{1}$; and so on.

\subsection{Illustration for CIR and WF signals}\label{sec: illustration}

In order to appreciate the ideas behind the main theoretical results
and the Algorithms we develop in this article, we provide some intuition on the 
corresponding results for one-dimensional hidden Markov models based on Cox--Ingersoll--Ross (CIR) and Wright--Fisher (WF) 
signals. These are the one-dimensional projections of the DW and FV
processes respectively, so informally we could say that a CIR process stands to a DW process as a gamma distribution stands to a gamma random measure, and a one-dimensional WF stands to a FV process as a Beta distribution stands to a Dirichlet process. The results illustrated in this section follow from \cite{PR14} and are based on the interplay between computable filtering and duality of Markov processes, summarised later in Section \ref{sec: filtering and
  duality}. The developments in this article rely on these
results, which are extended to the infinite-dimensional case.  
Here we highlight the mechanisms underlying the
explicit filters with the aid of figures, and postpone the mathematical details to Section \ref{sec: computable filtering}.

First, let the signal be a one-dimensional Wright--Fisher diffusion on
[0,1], with stationary distribution $\pi=\text{Beta}(\alpha,\beta)$ (see Section \ref{subsec: FV-model}) which is also taken as the prior $\nu$ for the signal at time 0. The signal can be interpreted as the evolving frequency of type-1 individuals in a
population of two types whose individuals generate offspring of the parent type which may be subject to mutation. 
The observations are assumed to be Bernoulli with
success probability given by the signal. Upon observation of $\yy_{t_{0}}=(y_{t_{0},1},\ldots,y_{t_{0},m})$, assuming it gives $m_1$ type-1 and $m_2$ type-2 individuals with $m=m_{1}+m_{2}$, the prior $\nu=\pi$ is updated  as usual via
Bayes' theorem to $\nu_{0}=\phi_{\yy_{t_{0}}}(\nu)=\text{Beta}(\alpha+m_{1},\beta+m_{2})$. Here $\phi_{\yy}$ is the update operator \eqref{update operator}. A forward propagation of these distribution of time $t$ by means of the prediction operator \eqref{prediction operator} yields the finite mixture of Beta distributions
\begin{equation*}
\begin{aligned}&\,
\psi_{t}(\nu_{0})=\sum_{(0,0)\le (i,j)\le (m_{1},m_{2})}p_{(m_{1},m_{2}),(i,j)}(t)\text{Beta}(\alpha+i,\beta+j),
\end{aligned}
\end{equation*} 
whose mixing weights depend on $t$ (see Lemma \ref{lemma: death transitions probabilities} below for their precise definition). The propagation of $\text{Beta}(\alpha+m_{1},\beta+m_{2})$ at time $t_{0}+t$ thus yields a mixture of Beta's with 
$(m_{1}+1)(m_{2}+1)$ components. The Beta parameters
range from $i=m_1,j=m_2$, which represent the full information provided by the collected data, to $i=j=0$,  which represent the null information on the data so that the associated component coincides with the prior. It is useful to identify the indices of the mixture with the nodes of a graph, as in  Figure \ref{fig: graph 1}-(b), where the red node represent the component with full information, and the yellow nodes the other components, including the prior identified by the origin. 
\begin{figure}[t!]
\begin{center}
\subfigure[]{
\raisebox{17.5mm}{\begin{tikzpicture}[scale=.35]
 \GraphInit[vstyle=Shade]
 \SetGraphShadeColor{yellow}{black}{yellow}
 \tikzset{LabelStyle/.style= {draw,
                              fill  = yellow,
                              text  = black}}
\SetGraphUnit{3.2}
\tikzset{VertexStyle/.style = {shape = circle,
                                shading = ball,
                                ball color =yellow,
                                minimum size = 18pt,
                                inner sep = 1pt,
                                draw,
			       color=black}}			       
\Vertex[L=\text{$0$}]{0}
\EA[L=\text{$1$}](0){1}
\EA[L=\text{$2$}](1){2}
\tikzset{VertexStyle/.style = {shape = circle,
                                shading = ball,
                                ball color = red,
                                minimum size = 18pt,
                                inner sep = 1pt,
                                draw,
			       color=black}}
\EA[L=\text{$3$}](2){3}
\tikzset{VertexStyle/.style = {shape = circle,
                                shading = ball,
                                ball color = white,
                                minimum size = 18pt,
                                inner sep = 1pt,
                                draw,
			       color=black}}
\EA[L=\text{$4$}](3){4}
\tikzset{EdgeStyle/.style = {->, ultra thick}}
\Edge(1)(0)
\Edge(2)(1)
\Edge(3)(2)
\Edge(4)(3)
\tikzset{VertexStyle/.style = {}}
 \end{tikzpicture}}}
 \hspace{4mm}
 \subfigure[]{
  \hspace*{-3mm}
  \begin{tikzpicture}[scale=.35]
 \GraphInit[vstyle=Shade]
 \SetGraphShadeColor{yellow}{black}{yellow}
 \tikzset{LabelStyle/.style= {draw,
                              fill  = yellow,
                              text  = black}}
\SetGraphUnit{2.5}
\tikzset{VertexStyle/.style = {shape = circle,
                                shading = ball,
                                ball color =red,
                                minimum size = 18pt,
                                inner sep = 1pt,
                                draw,
			       color=black}}			       
\Vertex[L=\text{$2,1$}]{21}
\tikzset{VertexStyle/.style = {shape = circle,
                                shading = ball,
                                ball color = yellow,
                                minimum size = 18pt,
                                inner sep = 1pt,
                                draw,
			       color=black}}
\SOWE[L=\text{$2,0$}](21){20}
\SOEA[L=\text{$1,1$}](21){11}
\SOEA[L=\text{$0,1$}](11){01}
\SOEA[L=\text{$1,0$}](20){10}
\SOWE[L=\text{$0,0$}](01){00}
\tikzset{VertexStyle/.style = {shape = circle,
                                shading = ball,
                                ball color = white,
                                minimum size = 18pt,
                                inner sep = 1pt,
                                draw,
			       color=black}}
\NOEA[L=\text{$2,2$}](21){22}
\NOEA[L=\text{$1,2$}](11){12}
\NOEA[L=\text{$0,2$}](01){02}
\tikzset{EdgeStyle/.style = {->, ultra thick}}
\Edge(22)(21)
\Edge(22)(12)
\Edge(21)(20)
\Edge(21)(11)
\Edge(12)(11)
\Edge(12)(02)
\Edge(20)(10)
\Edge(11)(10)
\Edge(11)(01)
\Edge(10)(00)
\Edge(01)(00)
\Edge(02)(01)
\tikzset{VertexStyle/.style = {}}
 \end{tikzpicture}}
\linebreak
\caption{\scriptsize  The death process on the lattice modulates the
  evolution of the mixture weights in the filtering distributions of
  models with CIR (left) and WF (right) signals. Nodes on the graph identify mixture components in the filtering distribution. The mixture weights are assigned according to the probability that the death process  starting from the (red) node which encodes the current full information (here $y=3$ for the CIR and $(m_{1},m_{2})=(2,1)$ for the WF) is in a lower node after time $t$.}\label{fig: graph 1}
\end{center}
\end{figure}
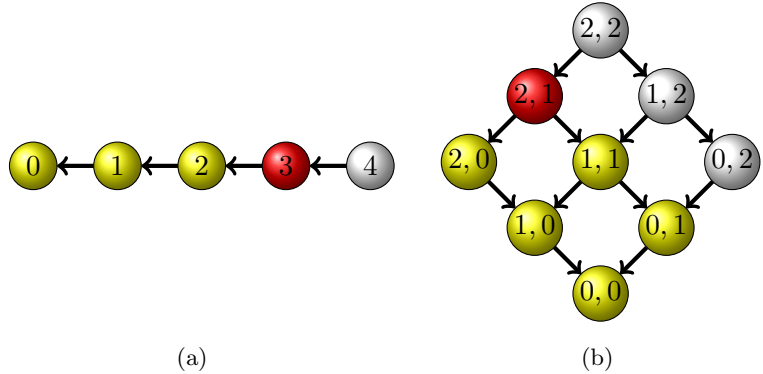
The time-varying mixing weights are the transition probabilities of an associated (dual) 2-dimensional death
process, which can be thought of as jumping to lower nodes in  the graph of Figure \ref{fig: graph 1}-(b) at a specified rate in continuous time. The effect on the mixture of these weights is that as time increases, the probability mass is shifted from components with
parameters close to the full information $(\alpha+m_{1},\beta+m_{2})$,
to components which bear less to none information on the data. The mass shift reflects the progressive obsolescence of the data collected at $t_{0}$ as evaluated by signal law at time $t_{0}+t$ as t increases. 

Note that it is not obvious that \eqref{prediction operator} yields a finite mixture when $P_{t}$ is the transition operator of a WF process, since $P_{t}$ has an infinite series expansion (see Section \ref{subsec: FV-model}). This has been proved rather directly in \cite{CG09} or by combining results on optimal filtering with some duality properties of this model (see \cite{PR14} or Section \ref{sec: computable filtering} here).

Consider now the model where the signal is  a one-dimensional CIR diffusion on
$\R_{+}$, with gamma stationary distribution (and prior at $t_{0}=0$) given by  $\pi=\text{Ga}(\alpha,\beta)$ (see Section \ref{subsec: DW-model}).
The observations are Poisson with intensity given by the current state of the signal. Before data are collected, the forward propagation of the signal distribution to time $t_{1}$ yields the same distribution by stationarity. Upon observation, at time $t_{1}$, of $m \geq 1$ Poisson data points with total
count $y$, the prior $\nu=\pi$ is  updated via Bayes' theorem
to
\begin{equation}\label{posterior gamma}
\nu_{0}=\text{Ga}(\alpha+y,\beta+m)
\end{equation} 
 yielding a jump in the measure-valued process; see Figure \ref{fig: cir3d+mixture}-(a). A forward propagation of $\nu_{0}$ yields the finite mixture of gamma distributions
\begin{equation}\label{gamma mixture}
\psi_{t}(\nu_{0})=\sum_{0\le i\le y}p_{y,i}(t)\text{Ga}(\alpha+i,\beta+\SS_{t}).
\end{equation} 
whose mixing weights also depend on $t$ (see Lemma \ref{lemma: 1-CIR propagation} below for their precise definition). 
At time $t_{1}+t$, the filtering distribution is 
a $(y+1)$-components mixture with the first gamma parameter ranging
from  full ($i=y$) to null ($i=0$) information with respect to the collected data (Figure \ref{fig: cir3d+mixture}-(b)).  
The time-dependent mixture weights are the transition probabilities of a a certain associated (dual)  one-dimensional death
process, which can be thought of as jumping to lower nodes in the graph of Figure \ref{fig: graph 1}-(a) at a specified rate in continuous time.
\begin{figure}[h!]
\begin{center}
\subfigure[]{
\hspace{-5mm}
\includegraphics[width=.4\textwidth]{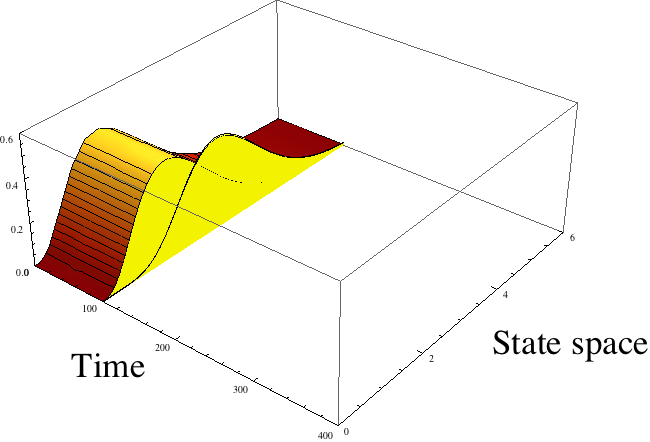}
\hspace{5mm}
\raisebox{7mm}{\includegraphics[width=.35\textwidth]{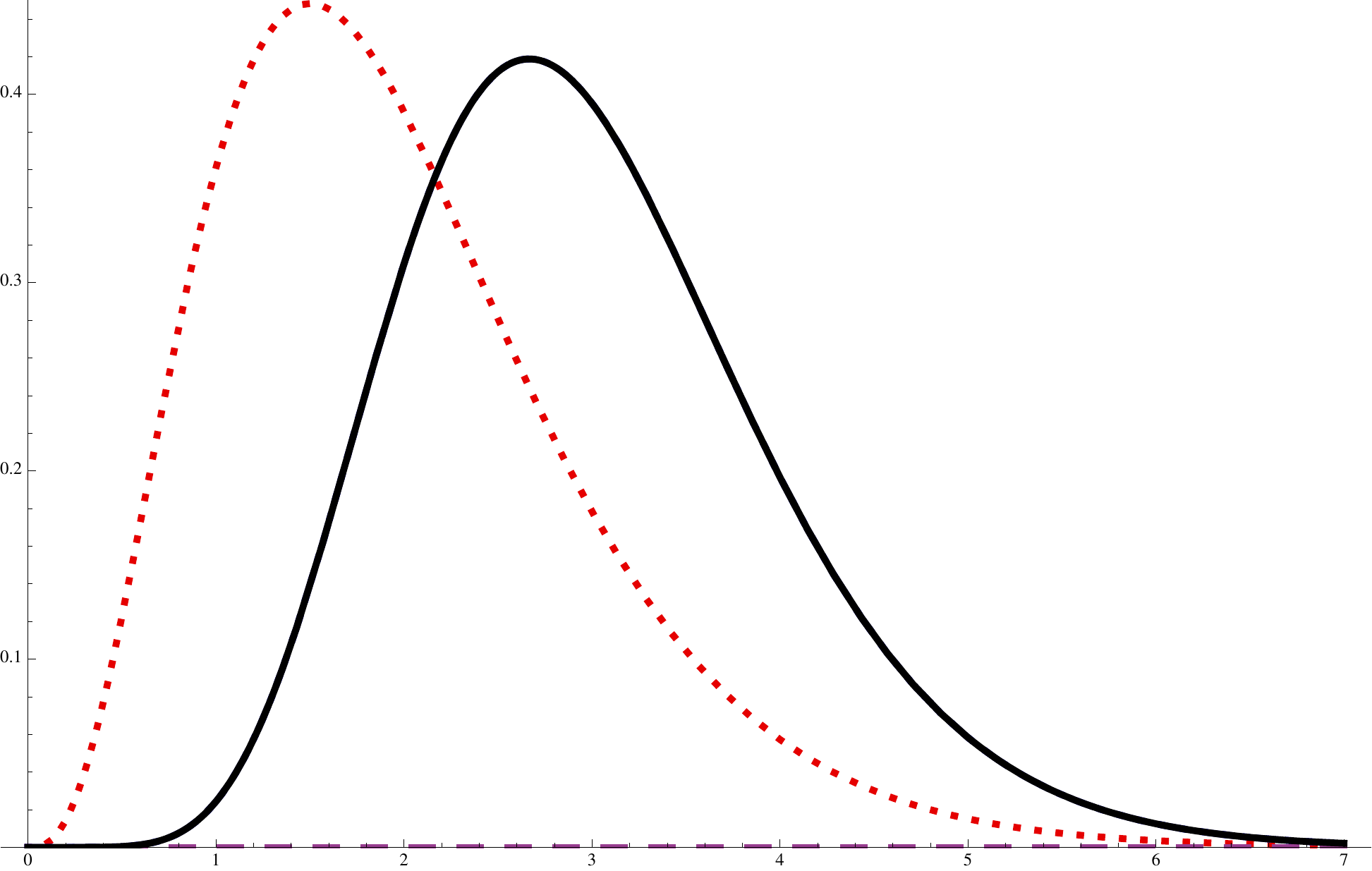}}
}\\\vspace{-3mm}
\subfigure[]{
\hspace{-5mm}
\includegraphics[width=.4\textwidth]{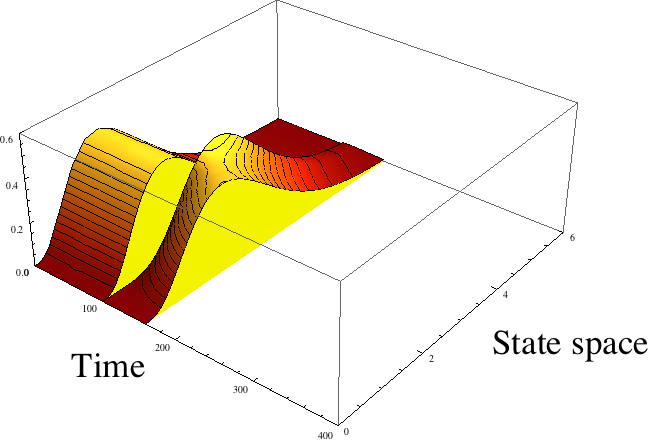}
\hspace{5mm}
\raisebox{7mm}{\includegraphics[width=.35\textwidth]{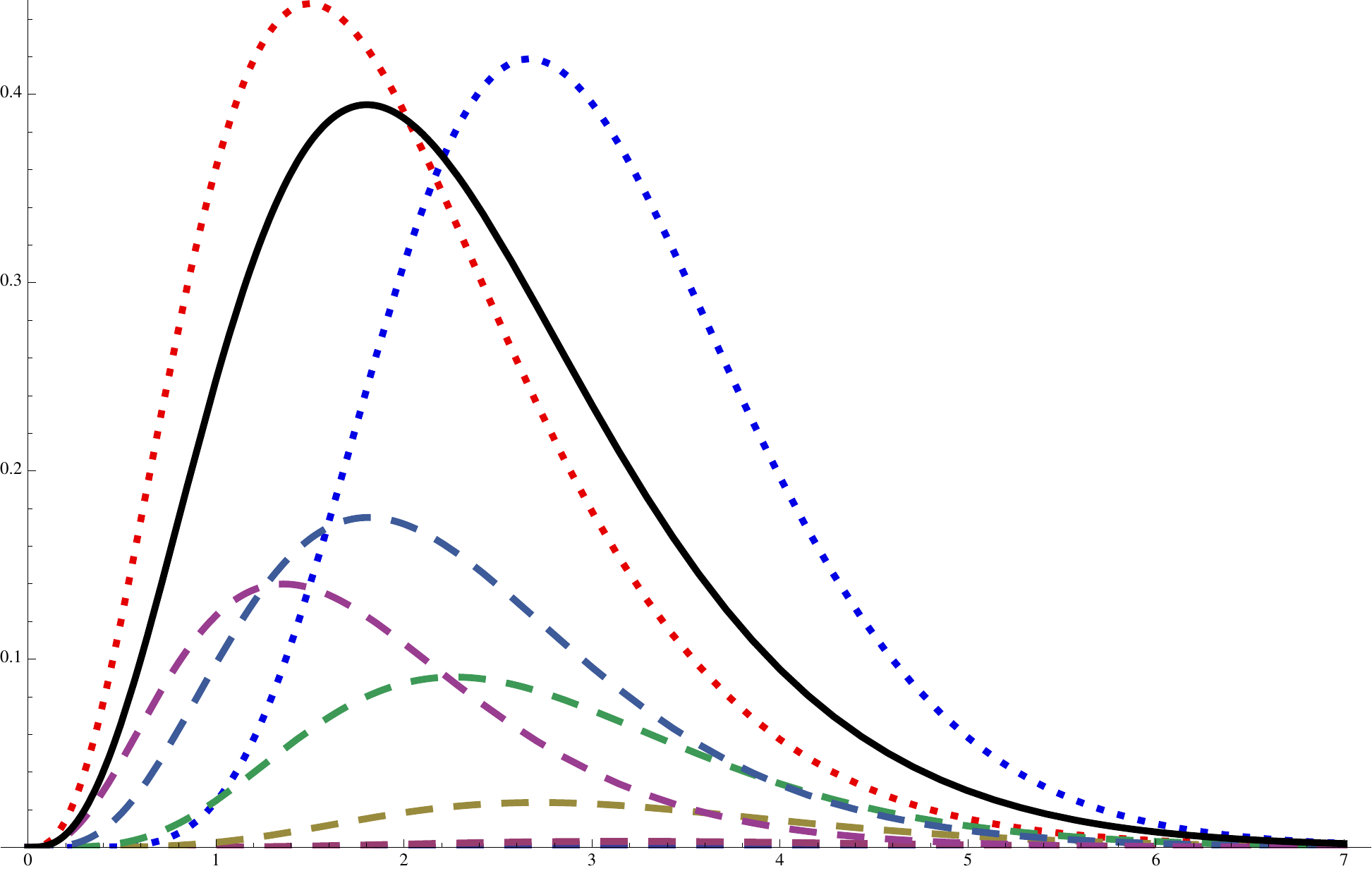}}
}\\\vspace{-3mm}
\subfigure[]{
\hspace{-5mm}
\includegraphics[width=.4\textwidth]{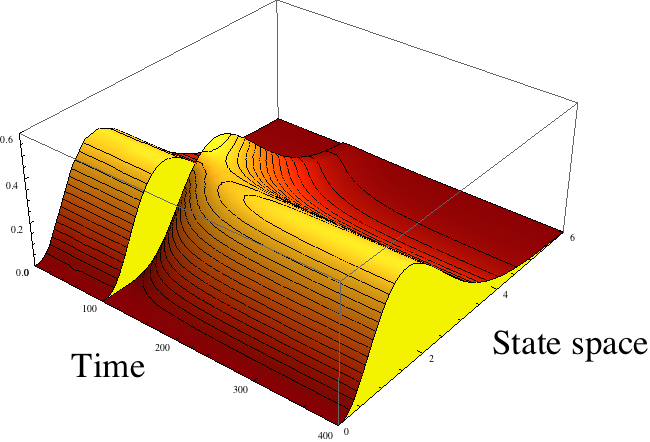}
\hspace{5mm}
\raisebox{7mm}{\includegraphics[width=.35\textwidth]{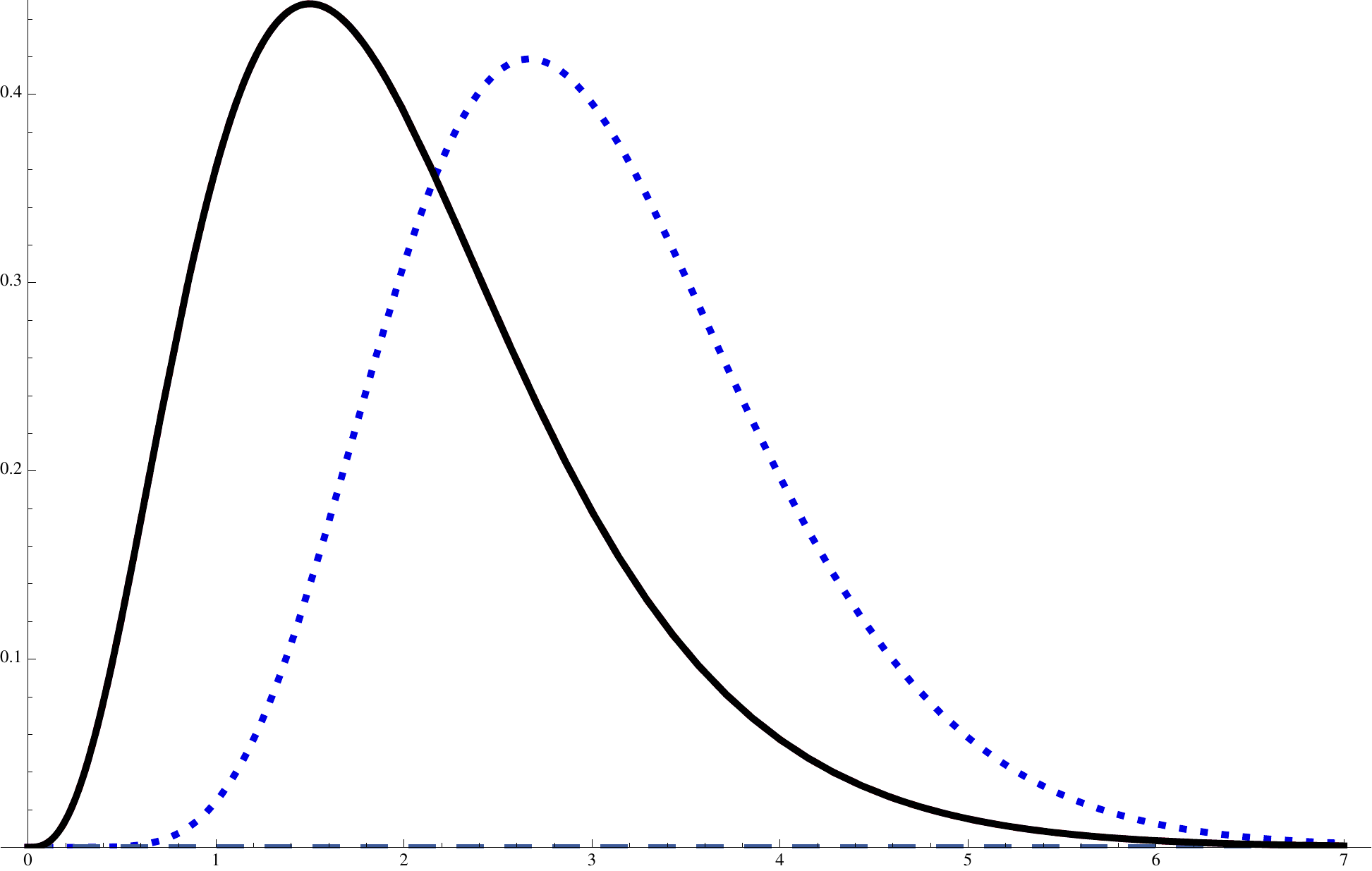}}
}
\end{center}
\vspace{-5mm}
\caption{\scriptsize 
Temporal evolution of the filtering distribution (solid black in right panels and marginal rightmost section of left panels) under the CIR model: 
(a) until the first data collection the propagation preserves the prior/stationary distribution (red dotted in right panels); at the first data collection, the prior is updated to the posterior (blue dotted in right panels) via Bayes' Theorem, determining a jump in the filtering process (left panel);
(b) the forward propagation of the filtering distribution behaves as a finite mixture of Gamma densities (weighted components dashed coloured in right panel);
(c) in absence of further data, the time-varying mixing weights shift mass towards the prior component, and the filtering distribution eventually returns to the stationary state.
}\label{fig: cir3d+mixture}
\end{figure}
Similarly to the WF model, the mixing weights shift mass from components whose first parameter is close to the full information, i.e.~$(\alpha+y,\beta+\SS_{t})$, to components which bear less to none
information $(\alpha,\beta+\SS_{t})$. The time evolution of the mixing weights is depicted in Figure \ref{fig: trans prob}, where the cyan and blue lines are the weights of the components with full and no information on the data respectively. 
\begin{figure}[t!]
\begin{center}
\includegraphics[width=.52\textwidth]{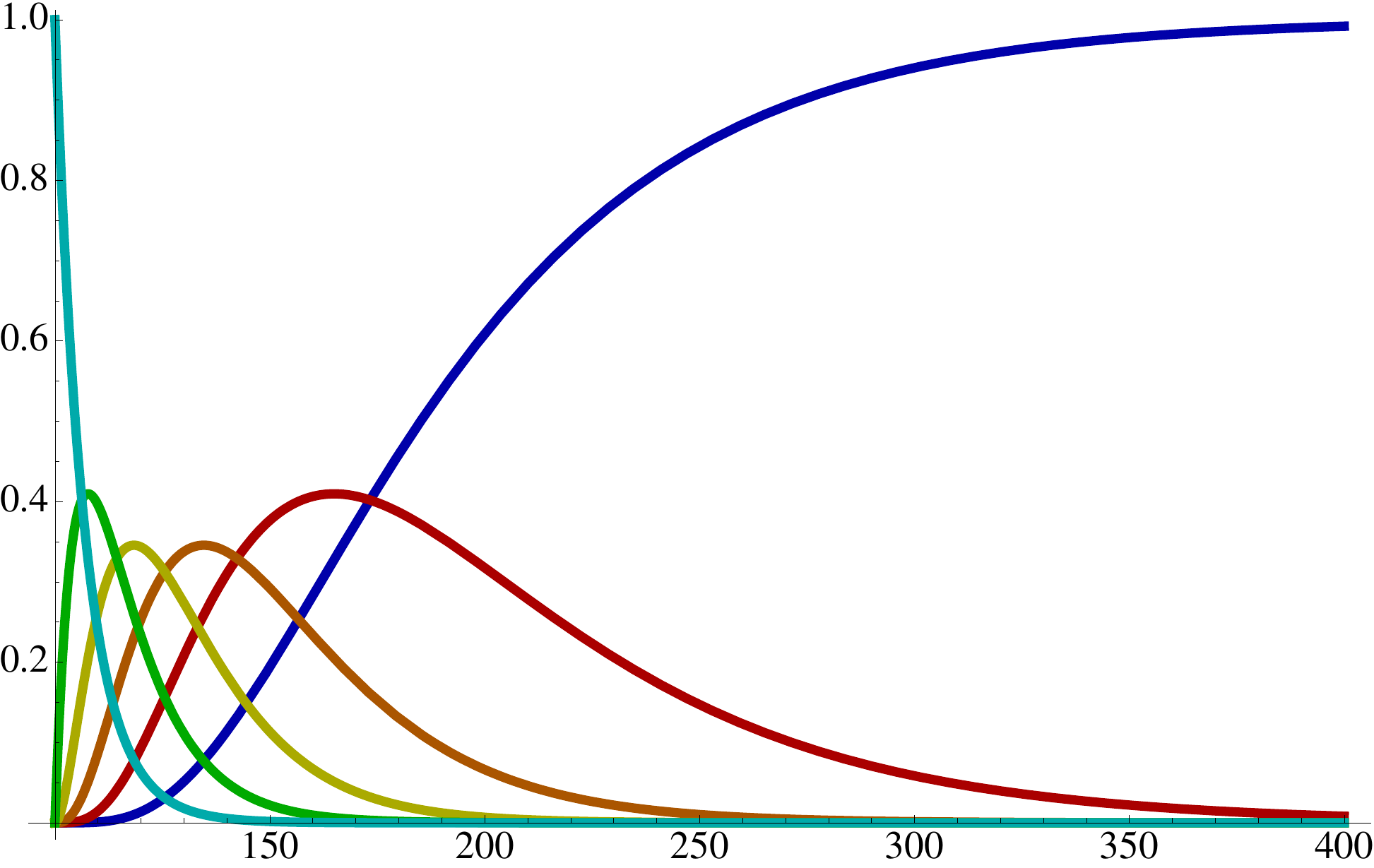}
\end{center}
\caption{\scriptsize 
Evolution of the mixture weights which drive the mixture distribution in Fig.~\ref{fig: cir3d+mixture}. At the jump time 100 (the origin here), the mixture component with full posterior information (blue dotted in Fig.~\ref{fig: cir3d+mixture}) has weight equal to 1 (cyan curve), and the other components have zero weight. As the filtering distribution is propagated forward, the weights evolve as transition probabilities of an associated death process. The mixture component equal to the prior distribution (red dotted in Fig.~\ref{fig: cir3d+mixture}), which carries no information on the data, has weight (blue curve) that is 0 at the jump time when the posterior update occurs, and eventually goes back to 1 in absence of further incoming observations, in turn determining the convergence of the mixture to the prior in Fig.~\ref{fig: cir3d+mixture}.
}\label{fig: trans prob}
\end{figure}
As a result of the impact of these weights on the mixture, the latter converges, in absence of further data, to the prior/stationary distribution $\pi$ as $t$ increases, as shown in Figure \ref{fig: cir3d+mixture}-(c). 
Unlike the WF case, in this model there is a second parameter controlled by a deterministic (dual) process $\SS_{t}$ on $\R_{+}$ which subordinates the transitions of the death process; see Lemma \ref{lemma: 1-CIR propagation}. Roughly speaking, the death process on the graph controls the obsolescence of the observation counts $y$, whereas the deterministic  process $\SS_{t}$ controls that of the sample size $m$. At the update time $t_{1}$ we have $\SS_{0}=m$ as in \eqref{posterior gamma}, but $\SS_{t}$ is a deterministic, continuous and decreasing process, and in absence of further data $\SS_{t}$ converges to 0 as $t\rightarrow \infty$, to restore the prior parameter $\beta$ in the limit of \eqref{gamma mixture}.
See Lemma \ref{lemma: 1-CIR propagation} in the Appendix for the formal result for the one-dimensional CIR diffusion.  

\begin{minipage}{\textwidth}
When more data samples are collected at different times, the update and propagation operations are alternated, resulting in jump processes for both the filtering distribution and the deterministic dual $\SS_{t}$ (Figure \ref{fig: Theta_t}). 
\end{minipage}

\begin{figure}[h!]
\begin{center}
\vspace{5mm}
\subfigure[]{\includegraphics[width=.4\textwidth]{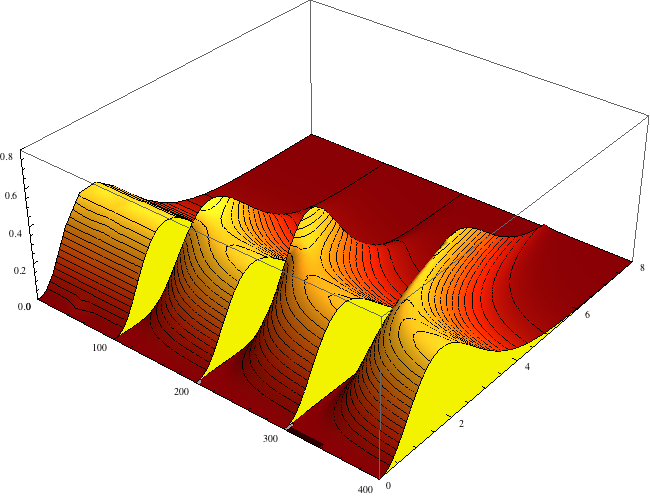}}
\subfigure[]{\raisebox{7mm}{\includegraphics[width=.35\textwidth]{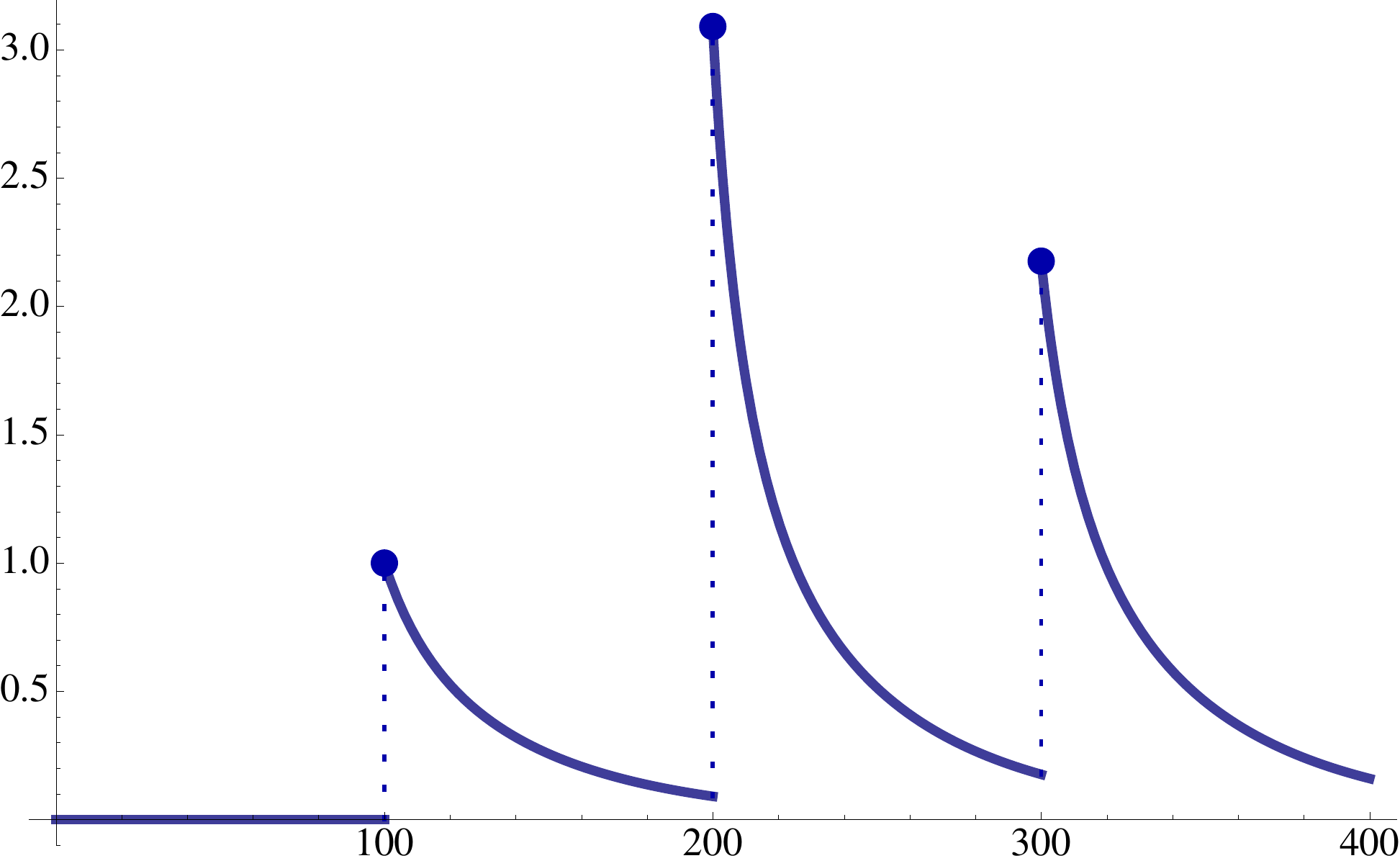}}}
\end{center}
\caption{\scriptsize 
Evolution of the filtering distribution (a) and of the deterministic component $\Theta_{t}$ of the dual process $(M_{t},\Theta_{t})$ that modulates the sample size parameter in the mixture components, in the case of multiple data collection at time $100,200,300$.
}\label{fig: Theta_t}
\end{figure}


\subsection{Preliminary notation}\label{sec: notation}

Although most of the notation is better introduced in the appropriate places, we collect here that which is used uniformly over the paper, to avoid recalling these objects several times throughout the text. 
In all subsequent sections, $\Y$ will denote a locally compact Polish space which represents the observations space, $\MM(\Y)$ is the associated space of finite Borel measures on $\Y$ and $\MM_{1}(\Y)$ its subspace of probability measures. A typical element $\alpha\in \MM(\Y)$ will be such that
\begin{equation}\label{alpha measure}
\alpha=\theta P_{0}, \quad
\theta>0,\quad
P_{0}\in\MM_{1}(\Y),
\end{equation}
where $\theta=\alpha(\Y)$ is the total mass of $\alpha$, and $P_{0}$ is sometimes called centering or baseline distribution. We will assume here that $P_{0}$ has no atoms. Furthermore, for $\alpha$ as above, $\Pi_{\alpha}$ will denote the law on $\MM_{1}(\Y)$ of a Dirichlet process, and $\Gamma^{\beta}_{\alpha}$ that on $\MM(\Y)$ of a gamma random measure, with $\beta>0$. These will be recalled formally in Sections \ref{subsec: FV-static} and \ref{subsec: DW-static}. 

We will denote by $X_{t}$ the Fleming--Viot process and by $Z_{t}$ the Dawson--Watanabe process, to be interpreted as $\{X_{t},t\ge0\}$ and $\{Z_{t},t\ge0\}$ when written without argument. Hence $X_{t}$ and $Z_{t}$ take values in the space of continuous functions from $[0,\infty)$ to $\MM_{1}(\Y)$ and $\MM(\Y)$ respectively, whereas discrete measures $x(\cdot)\in\MM_{1}(\Y)$ and $z(\cdot)\in\MM(\Y)$ will denote the marginal states of  $X_{t}$ and $Z_{t}$. We will write $X_{t}(A)$ and $Z_{t}(A)$ for their respective one dimensional projections onto the Borel set $A\subset\Y$. We adopt boldface notation to denote vectors, with the following conventions:
\begin{equation*}\label{vector notation}
\begin{split}
&\xx=(x_{1},\ldots,x_{K})\in\R_{+}^{K}, \quad
\mm=\,(m_{1},\ldots,m_{K})\in \Z^{K},\\
&\xx^{\mm}=x_{1}^{m_{1}}\cdots x_{K}^{m_{K}},\quad\quad \hspace{3mm}
|\xx|=\sum\nolimits_{i=1}^{K}x_{i},
\end{split}
\end{equation*}
where the dimension $2\le K\le \infty$ will be clear from the context unless specified. Accordingly, the Wright--Fisher model, closely related to projections of the Fleming--Viot process onto partitions, will be denoted $\XX_{t}$. 
We denote by $\oo$ the vector of zeros and by $\ee_{i}$ the vector whose only non zero entry is a 1 at the $i$th coordinate. Let also ``$<$'' define a partial ordering on $\Z^K$, so that $\mm<\nn$ if  $m_{j}\le n_{j}$ for all $j\ge1$ and $m_{j}< n_{j}$ for some $j\ge1$. Finally, we will use the compact notation $\yy_{1:m}$ for vectors of observations $y_{1},\ldots,y_{m}$.


\section{Hidden Markov measures}
\label{sec:hmmeas}

\subsection{Fleming--Viot signals}\label{sec: FV}

\subsubsection{The static model: Dirichlet processes and mixtures thereof}\label{subsec: FV-static}

The Dirichlet process on a state space $\Y$, introduced by \cite{F73}
(see \cite{G10} for a recent review), is a discrete random probability
measure $x\in\MM_{1}(\Y)$. 
The process admits the series representation
\begin{equation}\label{DP series representation}
x(\cdot)=\sum_{i=1}^{\infty}W_{i}\delta_{Y_{i}}(\cdot), \quad
W_{i}=\frac{Q_{i}}{\sum_{j\ge1}Q_{j}},\quad
Y_{i}\overset{iid}{\sim}P_{0},
\end{equation}
where $(Y_i)$ and $ (W_i)$ are independent and  $(Q_{i})$ are the jumps of a gamma process
with mean measure $\na y^{-1}e^{-y}\d y$. 
We will denote by $\Pi_{\alpha}$, the law of
$x(\cdot)$ in \eqref{DP series representation}, with $\alpha$ as in \eqref{alpha measure}.

Mixtures of Dirichlet processes were introduced in \cite{A74}. We say that $x$ is a mixture of Dirichlet processes if
\begin{equation*}
x\mid u\sim\Pi_{\alpha_{u}},\quad \quad
u\sim H,
\end{equation*}
where $\alpha_{u}$ denotes the measure $\alpha$ conditionally on $u$, or equivalently
\begin{equation}\label{MDP}
x\sim \int_{\mathcal{U}}\Pi_{\alpha_{u}}\d H(u).
\end{equation}
With a slight abuse of terminology we will also refer to the right hand side of the last expression as a mixture of Dirichlet processes.

The Dirichlet process and mixtures thereof have two fundamental properties that are of great interest in statistical learning \citep{A74}: 
\begin{itemize}
\item \emph{Conjugacy}:
let $x$ be as in \eqref{MDP}. Conditionally on $m$ observations $y_{i}\mid x\overset{iid}{\sim}x$, we have
\begin{equation*}
x\mid \yy_{1:m}
\sim \int_{\mathcal{U}}\Pi_{\alpha_{u}+\sum_{i=1}^{m}\delta_{y_{i}}}\d H_{\yy_{1:m}}(u),
\end{equation*}
where $H_{\yy_{1:m}}$ is the conditional distribution of $u$
given $\yy_{1:m}$. Hence a posterior mixture of Dirichlet processes is still a mixture of Dirichlet processes with updated parameters. 

\item \emph{Projection}: let $x$ be as in \eqref{MDP}. For any measurable partition $A_{1},\ldots,A_{K}$ of $\Y$, we have
  \begin{equation*}
(x(A_{1}),\ldots,x(A_{K}))
\sim
\int_{\mathcal{U}}\pi_{\aa_{u}}\d H(u),
\end{equation*}
where $\aa_{u}=(\alpha_{u}(A_{1}),\ldots,\alpha_{u}(A_{K}))$ and
$\pi_{\aa}$ denotes the Dirichlet distribution with parameter $\aa$.
\end{itemize}

Letting $H$ be concentrated on a single point of $\mathcal{U}$  recovers the respective properties of the Dirichlet
process as special case, i.e.~$x\sim \Pi_{\alpha}$ and $y_{i}|x\overset{iid}{\sim}x$ imply respectively that $x|\yy_{1:m}\sim \Pi_{\alpha+\sum_{i=1}^{m}\delta_{y_{i}}}$ and $(x(A_{1}),\ldots,x(A_{K}))\sim\pi_{\aa}$, where $\aa=(\alpha(A_{1}),\ldots,\alpha(A_{K}))$.

\subsubsection{The Fleming--Viot process}\label{subsec: FV-model}

Fleming--Viot (FV) processes are a large family of diffusions taking values in the subspace of $\MM_{1}(\Y)$ given by purely atomic probability measures. Hence they describe evolving discrete distributions whose support also varies with time and whose frequencies are each a diffusion on $[0,1]$. Two states apart in time of a FV process are depicted in Figure \ref{fig: FV}.  See \cite{EK93} and \cite{D93} for exhaustive reviews. Here we restrict the attention to a subclass known as the (labelled) \emph{infinitely many neutral alleles model} with parent independent mutation, henceforth for simplicity called the FV process, which has the law of a Dirichlet process as stationary measure \citep[Section 9.2]{EK93}.

\begin{figure}[t!]
\begin{center}
\subfigure[]{\includegraphics[width=.5\textwidth]{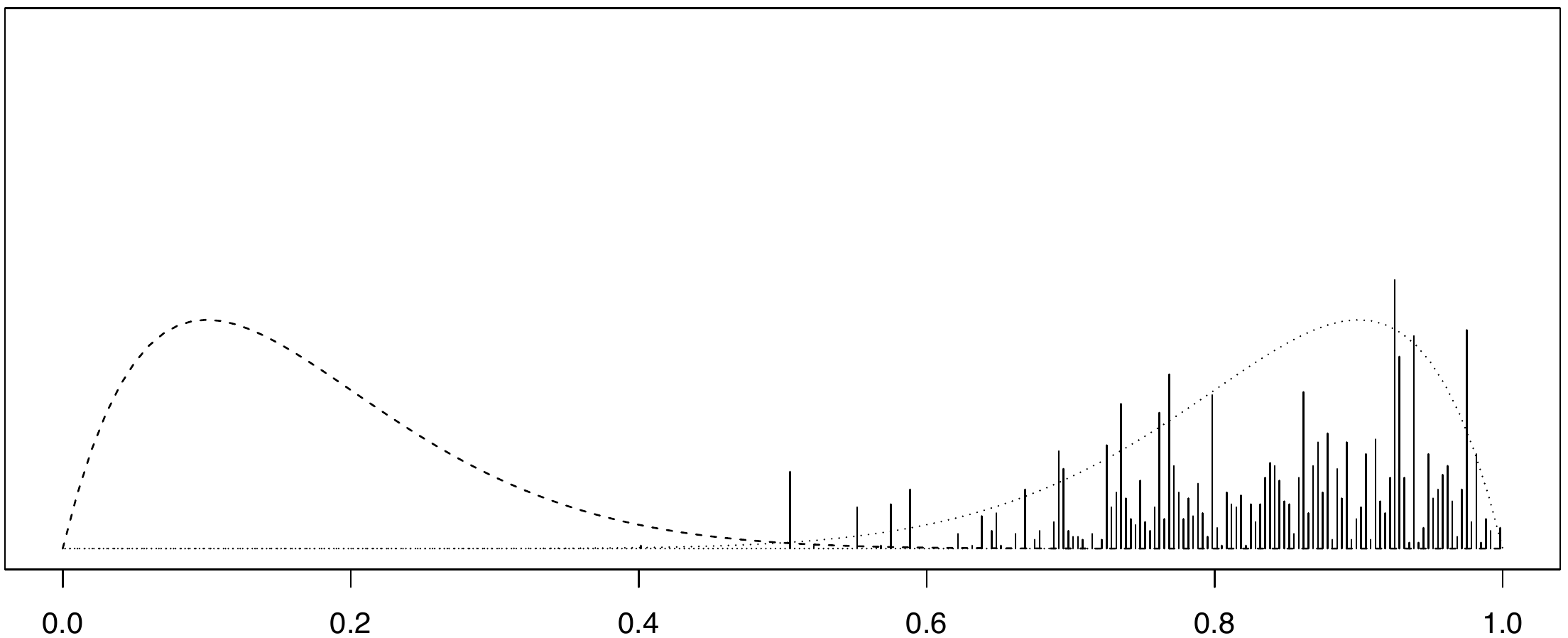}}\\
\subfigure[]{\includegraphics[width=.5\textwidth]{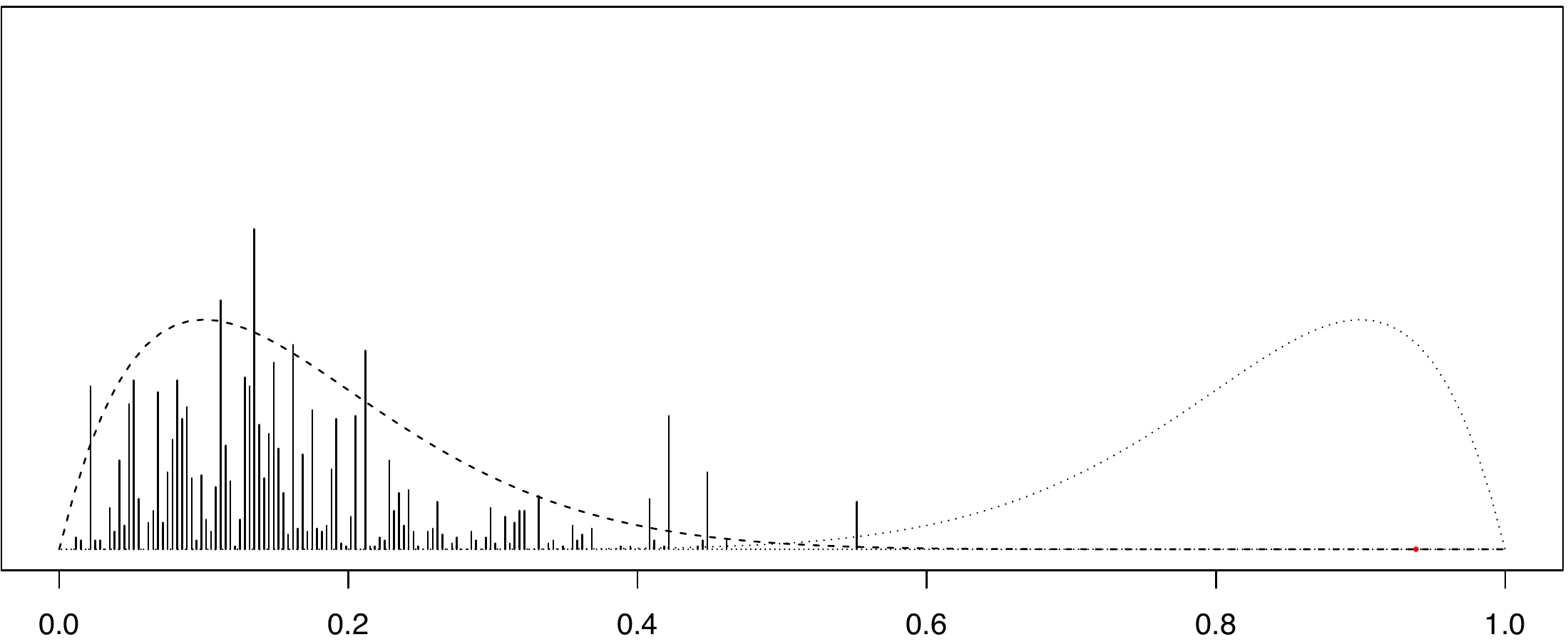}}
\end{center}
\begin{quote}
\caption{\scriptsize 
Two states of a FV process on $[0,1]$ at successive times (solid discrete measures): (a) the initial state has distribution $\Pi_{\alpha_{0}}$ with $\alpha_{0}=\theta  Beta(4,2)$ (dotted); (b) after some time, the process reaches the stationary state, which has distribution $\Pi_{\alpha}$ with $\alpha=\theta Beta(2,4)$ (dashed). 
}\label{fig: FV}
\end{quote}
\end{figure}

One of the most intuitive ways to understand a FV process is to consider its transition function, found in \cite{EG93}. This is given by 
\begin{equation}\label{FV trans-funct}
P_{t}(x,\d x')
=\sum_{m=0}^{\infty}d_{m}(t)\int_{\Y^m}
\Pi_{\alpha+\sum_{i=1}^{m}\delta_{y_{i}}}
(\d x')x^{m}(\d
y_1,\ldots,\d y_m)
\end{equation}
where $x^{m}$ denotes the $m$-fold product measure $x\times\dots\times x$ and $\Pi_{\alpha+\sum_{i=1}^{m}\delta_{y_{i}}}$ is a posterior Dirichlet process as defined in the previous section. The expression \eqref{FV trans-funct} has a nice interpretation from the Bayesian learning viewpoint. Given the current state of the process $x$, with probability $d_{m}(t)$ an $m$-sized sample from $x$ is taken, and the arrival state is sampled from the posterior law $\Pi_{\alpha+\sum_{i=1}^{m}\delta_{y_{i}}}$. Here $d_{m}(t)$ is the probability that an $\N$-valued death process which starts at infinity at time 0 is in $m$ at time $t$, if it jumps from $m$ to $m-1$ at rate $    \lambda_m=\frac{1}{2}m(\theta+m-1)$. See \cite{T84} for details. 
Hence a larger $t$ implies sampling a lower amount of information from $x$ with higher probability, resulting in fewer atoms shared by $x$ and $x'$. Hence the starting and arrival states have correlation which decreases in $t$ as controlled by $d_{m}(t)$. As $t\rightarrow 0$, infinitely many samples are drawn from $x$, and $x'$ will coincide with $x$, and the trajectories are continuous in total variation norm \citep{EK93}.  As $t\rightarrow \infty$, the fact that the death process which governs the probabilities $d_{m}(t)$ in \eqref{FV trans-funct} is eventually absorbed in 0 implies that $P_{t}(x, \d x')\rightarrow \Pi_{\alpha}$ as $t\rightarrow \infty$, so $x'$ is sampled from the prior $\Pi_{\alpha}$. Therefore this FV is stationary with respect to $\Pi_{\alpha}$ (in fact, it is also reversible). It follows that, using terms familiar to the Bayesian literature, under this parametrisation the FV can be considered as a dependent Dirichlet process with continuous sample paths. Constructions of Fleming--Viot and closely related processes using ideas from Bayesian nonparametrics have been proposed in \cite{Wea07,FRW09,RW09a,RW09b}. Different classes of diffusive  dependent Dirichlet processes or related are constructed in \cite{MR16,MRW11} based on the stick-breaking representation \citep{S94}.

Projecting a FV process $X_{t}$ onto a measurable partition $A_{1},\ldots,A_{K}$ of $\Y$ yields a $K$-dimensional Wright--Fisher (WF) diffusion $\XX_{t}$, which is reversible and stationary with respect to the Dirichlet distribution $\pi_{\aa}$, for $\alpha_{i}=\theta P_{0}(A_{i})$, $i=1,\ldots,K$. See \citep{D10,E09}.
This property is the dynamic counterpart of the projective property of Dirichlet processes discussed in Section \ref{subsec: FV-static}. Consistently, the transition function of a WF process is obtained as a specialisation of the FV case, yielding 
\begin{equation}\label{WF trans-funct}
P_{t}(\xx,\d \xx')
=\sum_{m=0}^{\infty}d_{m}(t)
\sum_{\mm\in\Z^{K}: \norm{\mm}=m}
\binom{m}{\mm}\xx^{\mm}
\pi_{\aa+\mm}(\d \xx')
\end{equation}
with analogous interpretation to \eqref{FV trans-funct}. See \cite{EG93}.

For statistical modelling it is useful to introduce a further
parameter $\sigma$ that controls the speed of the process. This can
be done by defining 
the time change $X_{\tau(t)}$ with
$\tau(t)=\sigma t$.  In such parameterisation,
$\sigma$ does not affect the stationary distribution of the process,
and can be used to model the dependence structure.


\subsection{Dawson--Watanabe signals}\label{sec: DW}

\subsubsection{The static model: gamma random measures and mixtures thereof}\label{subsec: DW-static}

Gamma random measures \citep{L82} can be thought of as the counterpart of Dirichlet processes in the context of finite intensity measures. A gamma random measure $z\in\MM(\Y)$ with shape parameter $\alpha$ as in \eqref{alpha measure} and rate parameter $\beta>0$, denoted $z\sim \Gamma^{\beta}_{\alpha}$, admits representation
\begin{equation}\label{gamma process series}
z(\cdot)=\beta^{-1}\sum_{i=1}^{\infty}Q_{i}\delta_{Y_{i}}(\cdot), \quad  Y_{i}\overset{iid}{\sim}P_{0},
\end{equation}
with $\{Q_{i}, i\ge1\}$ as in \eqref{DP series representation}.

Similarly to the definition of mixtures of Dirichlet processes (Section \ref{subsec: FV-static}), we say that $z$ is a mixture of gamma random measures if $z\sim \int_{\mathcal{U}}\Gamma_{\alpha_{u}}^{\beta}\d H(u)$, and with a slight abuse of terminology we will also refer to the right hand side of the last expression as a mixture of gamma random measures.
Analogous conjugacy and projection properties to those seen for mixtures of Dirichlet processes hold for mixtures of gamma random measures:

\begin{itemize}
\item \emph{Conjugacy}:
let $N$ be a Poisson point process on $\Y$ with random intensity measure $z$, i.e., conditionally on $z$, $N(A_i)\overset{ind}{\sim}\text{Po}(z(A_{i}))$ for any measurable partition $A_1,\ldots,A_K$ of $Y$, $K\in\N$. Let $m:=N(\Y)$, and given $m$, let $y_1,\ldots,y_m$ be a realisation of points of $N$, so that
\begin{equation}\label{poisson data}
y_{i}\mid z,m\overset{iid}\sim z/|z|,\quad \quad
m\mid z\sim\text{Po}(|z|)\quad \quad
\end{equation}
where  $|z|:=z(\Y)$ is the total mass of $z$.
 Then
\begin{equation}\label{posterior gamma RM}
z\mid \yy_{1:m}
\sim \int_{\mathcal{U}}\Gamma_{\alpha_{u}+\sum_{i=1}^{m}\delta_{y_{i}}}^{\beta+1}\d H_{\yy_{1:m}}(u),
\end{equation}
where $H_{\yy_{1:m}}$ is the conditional distribution of $u$
given $\yy_{1:m}$. Hence mixtures of gamma random measures are conjugate with respect to Poisson point process data.

\item \emph{Projection}: for any measurable partition $A_{1},\ldots,A_{K}$ of  $\Y$, we have
  \begin{equation*}
(z(A_{1}),\ldots,z(A_{K}))
\sim
\int_{\mathcal{U}}\prod_{i=1}^{K}\text{Ga}(\alpha_{u,i},\beta)\d H(u),
\end{equation*}
where $\alpha_{u,i}=\alpha_{u}(A_{i})$ and $\text{Ga}(\alpha,\beta)$  denote the gamma distribution with shape $\alpha$ and rate $\beta$.
\end{itemize}
Letting $H$ be concentrated on a single point of $\mathcal{U}$  recovers the respective properties of gamma random measures as special case, i.e.~$z\sim \Gamma_{\alpha}^{\beta}$ and $y_{i}$ as in \eqref{poisson data} imply $z|y_{1:m}\sim \Gamma^{\beta+1}_{\alpha+\sum_{i=1}^{m}\delta_{y_{i}}}$, and the vector $(z(A_{1}),\ldots,z(A_{K}))$ has independent components $z(A_{i})$ with gamma distribution $\text{Ga}(\alpha_{i},\beta)$, $\alpha_{i}=\alpha(A_{i})$.

Finally, it is well known that \eqref{DP series representation} and \eqref{gamma process series} satisfy the relation in distribution
\begin{equation}\label{gamma-dirichlet algebra}
x(\cdot)\overset{d}{=}\frac{z(\cdot)}{z(\Y)}
\end{equation}
where $x$ is independent of $z(\Y)$. This extends to the
infinite dimensional case the well known relationship between beta and
gamma random variables. See for example \cite{DVJ08}, Example 9.1(e). See also \cite{KS88} for an extension of \eqref{gamma-dirichlet algebra} to the dynamic case concerning FV and DW processes, which requires a random time change. 


\subsubsection{The Dawson--Watanabe process}\label{subsec: DW-model}

 Dawson--Watanabe (DW) processes can be considered as dependent models for gamma random measures, and are, roughly speaking, the gamma counterpart of FV processes. More formally, 
they are branching measure-valued diffusions taking values in the space of finite discrete measures. As in the FV case, they describe evolving discrete measures whose support  varies with time and whose masses are each a positive diffusion, but relaxing the constraint of their masses summing to one to that of summing to a finite quantity. See \cite{D93} and \cite{L11} for reviews.
Here we are interested in the special case of subcritical branching
with immigration, where subcriticality refers to the fact that in the underlying branching population which can be used to construct the process, the mean number of offspring per individual is less than one. Specifically, we will consider DW processes with transition function
\begin{equation}\label{DW trans-funct}
P_{t}(z,\d z')
=\sum_{m=0}^{\infty}d_{m}^{|z|,\beta}(t)\int_{\Y^m}
\Gamma^{\beta+S^*_{t}}_{\alpha+\sum_{i=1}^{m}\delta_{y_{i}}}
(\d z')(z/|z|)^{m}(\d y_1,\ldots,\d y_m).
\end{equation}
where $$d_{m}^{|z|,\beta}(t)=\text{Po}\left(m\ \Big\vert \frac{|z|\beta}{e^{\beta t/2}-1}\right)\ \ \ \ \text{and}\ \ \ \ S^*_{t}:=\frac{\beta} {e^{\beta t/2}-1}.$$
See \cite{EG93b}.  The interpretation of \eqref{DW trans-funct} is similar to that of \eqref{FV trans-funct}: conditional on the current state, i.e.~the measure $z$, $m$ \emph{iid}  samples are drawn from the normalised measure $z/|z|$ and the arrival state $z'$ is sampled from $\Gamma^{\beta+S^*_t}_{\alpha+\sum_{i=1}^{m}\delta_{y_{i}}}$. Here the main structural difference with respect to \eqref{FV trans-funct}, apart from the different distributions involved, is that since in general $S^{*}_{t}$ is not an integer quantity, the interpretation as sampling the arrival state $z'$ from a posterior gamma law is not formally correct; cf.~\eqref{posterior gamma RM}. 
The sample size $m$ is chosen with probability $d_{m}^{|z|,\beta}(t)$, which is the probability that an $\N$-valued death process which starts at infinity at time 0 is in $m$ at time $t$, if it jumps from $m$ to $m-1$ at rate $(m \beta/2)( 1-e^{\beta t/2})^{-1}$. See \cite{EG93b} for details.
So $z$ and $z'$ will share fewer atoms the farther they are apart in time. 
The DW process with the above transition is known to be stationary and reversible with respect to the law $\Gamma^{\beta}_{\alpha}$ of a gamma random measure; cf.~\eqref{gamma process series}. See \cite{S90,EG93b}.  The Dawson--Watanabe process has been recently considered as a basis to build time-dependent gamma process priors with Markovian evolution in \cite{CT12} and \cite{SL16}. 

The DW process satisfies a projective property similar to that seen in Section \ref{subsec: FV-model} for the FV process. 
Let $Z_{t}$ have transition \eqref{DW trans-funct}. Given a measurable partition $A_{1},\ldots,A_{K}$ of $\Y$, the vector $(Z_{t}(A_{1}),\ldots,Z_{t}(A_{K}))$ has independent components $z_{t,i}=Z_{t}(A_{i})$ each driven by a Cox--Ingersoll--Ross (CIR) diffusion \citep{CIR85}. These are also subcritical continuous-state branching processes with immigration, reversible and ergodic with respect to a $\text{Ga}(\alpha_{i},\beta)$ distribution, with transition function
\begin{equation}\label{CIR trans-funct}
P_{t}(z_{i},\d z_{i}')
=\sum_{m_i=0}^{\infty}
\text{Po}\bigg(m_i\,\Big\vert\, \frac{z_i \beta  }{e^{\beta t/2}-1}\bigg)
\mathrm{Ga}\bigg(\d z' \,\Big\vert\, \alpha_i+m_i, \beta+S^*_{t}\bigg).
\end{equation}
As for FV and WF processes, a further
parameter $\sigma$ that controls the speed of the process can be introduced without affecting the stationary distribution. 
This can be done by defining an appropriate time change that can be used to model the dependence structure.


\section{Conjugacy properties of time-evolving Dirichlet and gamma
  random measures}\label{sec: main results}


\subsection{Filtering Fleming--Viot signals}\label{sec: FV results}

Let the latent signal $X_{t}$ be a FV process with transition function \eqref{FV trans-funct}. We assume that, given the signal state, observations are drawn independently from $x$, i.e.~as in \eqref{DP data} with $X_{t}=x$.
Since $x$ is almost surely discrete \citep{B73}, a sample $\yy_{1:m}=(y_{1},\ldots,y_{m})$ from $x$ will feature $K_{m}\le m$ ties among the observations with positive probability. Denote by
$(y_{1}^{*},\ldots,y_{K_{m}}^{*})$ the distinct values in $\yy_{1:m}$ and by $\mm=(m_{1},\ldots,m_{K_{m}})$ the associated multiplicities, so that $\norm{\mm}=m$. When an additional sample $\yy_{m+1:m+n}$ with multiplicities $\nn$ becomes available, we adopt the convention that $\nn$ adds up to the multiplicities of the types already recorded in $\yy_{1:m}$, so that the total multiplicities count is
\begin{equation}\label{count update}
\mm+\nn=(m_{1}+n_{1},\ldots,m_{K_{m}}+n_{K_{m}},n_{K_{m}+1},\ldots,n_{K_{m+n}}).
\end{equation} 
The following Lemma states in our notation the special case of the conjugacy for mixtures of Dirichlet processes which is of interest here; see Section \ref{subsec: FV-static}. To this end, let
 \begin{equation}\label{set of multiplicities}
 \M=\{\mm=(m_1,\ldots,m_{K})\in\Z^{K},\, K\in\N\}
 \end{equation}
 be the space of multiplicities of $K$ types, with partial ordering defined as in Section \ref{sec: notation}. 
Denote also by $\mathrm{PU}_{\alpha}(\yy_{m+1:m+n}\mid \yy_{1:m})$ the joint distribution of $\yy_{m+1:m+n}$ given $\yy_{1:m}$ when the random measure $x$ is marginalised out, which follows the Blackwell--MacQueen P\'olya urn scheme \citep{BM73} 
\begin{equation*}
\begin{split}
&\,Y_{m+i+1} \mid \yy_{1:m+i} \sim
\frac{\theta P_{0}+\sum_{j=1}^{m+i}\delta_{y_{j}}}{\theta+m+i},\quad \quad i=0,\ldots,n-1.
\end{split}
\end{equation*}

\begin{lemma}\label{lem: FV update multi}
Let $M\subset \M$, $\alpha$ as in \eqref{alpha measure} and $x$ be the mixture of Dirichlet processes
\begin{equation*}
x\sim \sum_{\mm \in M}w_{\mm}\Pi_{\alpha+\sum_{i=1}^{K_{m}}m_{i}\delta_{y_{i}^{*}}},
\end{equation*}
with $\sum_{\mm\in M}w_\mm=1$. Given an additional $n$-sized sample $\yy_{m+1:m+n}$ from $x$ with multiplicities $\nn$, the update operator \eqref{update operator} yields
\begin{align}\label{eq: NP-update}
\phi_{\yy_{m+1:m+n}}\bigg(\sum_{\mm \in M}w_{\mm}\Pi_{\alpha+\sum_{i=1}^{K_{m}}m_{i}\delta_{y_{i}^{*}}}\bigg)
=
\sum_{\mm\in M}\hat w_{\mm}\Pi_{\alpha+\sum_{i=1}^{K_{m+n}}(m_{i}+n_{i})\delta_{y_{i}^{*}}},
\end{align}
where
\begin{equation}\label{updated mixture weights}
\hat w_{\mm}\propto
w_{\mm}\,\mathrm{PU}_{\alpha}(\yy_{m+1:m+n}\mid \yy_{1:m}).
\end{equation}
\end{lemma}
The updated distribution is thus still a mixture of Dirichlet processes with different multiplicities and possibly new atoms in the parameter measures $\alpha+\sum_{i=1}^{K_{m+n}}(m_{i}+n_{i})\delta_{y_{i}^{*}}$. 

The following Theorem formalises our main result on FV processes, showing that the family of finite mixtures of Dirichlet processes is conjugate with respect to discretely sampled data as in \eqref{DP data} with $X_{t}=x$.  
For $\M$ as in \eqref{set of multiplicities}, let
 \begin{equation}\label{L(M)}
\begin{split}
   L(\mm)=&\,\{\nn\in\M:\ \oo\le\nn\le\mm\},\quad \mm\in \M,\\
  L(M)=&\,\{\nn\in\M:\ \oo\le\nn\le\mm,\, \mm\in M\},\quad M\subset \M,
\end{split}
 \end{equation}
 be the set of nonnegative vectors lower than or equal to $\mm$ or to those in $M$ respectively, with ``$\le$'' defined as in Section \eqref{sec: notation}. For example, in Figure \ref{fig: graph 1}, $L(3)$ and $L((1,2))$ are both given by all yellow and red nodes in each case. Let also \begin{equation}\label{hypergeometric}
 p(\ii;\,\mm,|\ii|)=\binom{\norm{\mm} }{|\ii|}^{-1}
 \prod_{j\ge1}\binom{m_{j}}{i_{j}}
 \end{equation}
be the multivariate hypergeometric probability function, with parameters $(\mm,|\ii|)$, evaluated at $\ii$.

\begin{theorem}\label{thm: FV propagation}
Let $\psi_{t}$ be the prediction operator \eqref{prediction operator} associated to a FV process with transition operator \eqref{FV trans-funct}. Then the prediction operator yields as $t$-time-ahead propagation the finite mixture of Dirichlet processes
\begin{equation}\label{FV propagation}
\psi_{t}\bigg(\Pi_{\alpha+\sum_{i=1}^{K_{m}}m_{i}\delta_{y_{i}^{*}}}\bigg)=
\sum_{\nn\in L(\mm)}p_{\mm,\nn}(t)
\Pi_{\alpha+\sum_{i=1}^{K_{m}}n_{i}\delta_{y_{i}^{*}}},
\end{equation}
with $L(\mm)$ as in \eqref{L(M)} and where
 \begin{equation}\label{transition probabilities}
 p_{\mm,\mm-\ii }(t)=
\left\{
\begin{array}{ll}
e^{-\lambda_{\norm{\mm}}t},& \ii=\oo\\
  C_{\norm{\mm} ,\norm{\mm} -|\ii|}(t) p(\ii;\,\mm,|\ii|),
 \quad \quad &  \oo< \ii\le \mm,
\end{array}
\right. 
 \end{equation}
with
 \begin{equation*}
C_{\norm{\mm} ,\norm{\mm} -|\ii|}(t)=
\bigg(\prod_{h=0}^{|\ii|-1}\lambda_{\norm{\mm} -h}\bigg)
 (-1)^{|\ii|}\sum_{k=0}^{|\ii|}\frac{e^{-\lambda_{\norm{\mm} -k}t}}{\prod_{0\le h\le |\ii|,h\ne k}(\lambda_{\norm{\mm} -k}-\lambda_{\norm{\mm} -h})},
\end{equation*} 
$\lambda_{n}=n(\theta+n-1)/2$ and $ p(\ii;\,\mm,|\ii|)$ as in \eqref{hypergeometric}. 
\end{theorem}

The transition operator of the FV process thus maps a Dirichlet process at time $t_{0}$ into a finite mixture of Dirichlet processes at time $t_{0}+t$. The mixing weights are the transition probabilities of a death process on the $K_{m}$ dimensional lattice, with $K_{m}$ being as in \eqref{FV propagation} the number of distinct values in previous data. 
The result is obtained by means of the argument described in Figure \ref{fig: scheme}, which is based on
the property that the operations of propagating and projecting the signal
commute. By projecting the current distribution of the signal onto an
arbitrary measurable partition, yielding a mixture of Dirichlet
distributions, we can exploit the results for finite
dimensional WF signals to yield the associated propagation \citep{PR14}. The propagation of the original signal is then obtained by means of the characterisation of mixtures of Dirichlet processes via their
projections. See Section \ref{sec: filtering FV} for a proof. 
In particular, the result shows that under these
assumptions, the series expansion for the transition function
\eqref{FV trans-funct} reduces to a finite sum. 

Iterating the update and propagation operations provided by Lemma \ref{lem: FV update multi} and Theorem \ref{thm: FV propagation} allows to perform sequential Bayesian inference on a hidden signal of FV type by means of a finite computation. Here the finiteness refers to the fact that the infinite dimensionality due to the transition function of the signal is avoided analytically, without resorting to any stochastic truncation method for \eqref{FV trans-funct}, e.g.~\citep{W07,PR08}, and the computation can be conducted in closed form. 

The following Proposition formalises the recursive algorithm that sequentially evaluates the marginal posterior laws $\L(X_{t_n} | Y_{1:n})$ of a partially observed FV process by alternating the update and propagation operations, and identifies the family of distributions which is closed with respect to these operations. 
Define the family of finite mixtures of Dirichlet processes
\begin{align}\notag
\Fdir=\!
\Bigg\{\sum_{\mm\in M}w_{\mm}\Pi_{\alpha+\sum_{i=1}^{K_{m}}m_{i}\delta_{y_{i}^{*}}}\!:M \subset \M,|M|<\infty, w_{\mm}\ge0,\sum_{\mm\in M}w_\mm=1\Bigg\},
\end{align}
with $\M$ as in \eqref{set of multiplicities} and for a fixed $\alpha$ as in \eqref{alpha measure}. Define also
\begin{equation*}
t(\yy,\mm)=\mm+\nn,\quad \mm\in\Z^{K}
\end{equation*} 
so that $t(\yy,\mm)$ is \eqref{count update} if $\nn$ are the multiplicities of $\yy$, and
\begin{equation}\label{t map}
t(\yy,M) = \{\nn:\ \nn=t(\yy,\mm), \mm \in M \} , \quad M\subset \M.
\end{equation}

\begin{proposition}\label{prop: algorithm FV} 
Let $X_{t}$ be a FV process with transition function \eqref{FV trans-funct} and invariant law $\Pi_{\alpha}$ defined as in Section \ref{subsec: FV-static}, and suppose data are collected as in \eqref{DP data} with $X_{t}=x$. Then $\Fdir$ is closed under the application of the update and prediction operators \eqref{update operator} and \eqref{prediction operator}. Specifically,
\begin{align}\label{eq: FV algorithm update}
\phi_{\yy_{m+1:m+n}}\Bigg(\sum_{\mm \in M}w_{\mm}\Pi_{\alpha+\sum_{i=1}^{K_{m}}m_{i}\delta_{y_{i}^{*}}}\Bigg)
=
\sum_{\nn\in t(\yy_{m+1:m+n},M)}\hat w_{\nn}\Pi_{\alpha+\sum_{i=1}^{K_{m+n}}n_{i}\delta_{y_{i}^{*}}},
\end{align}
with $t(\yy,M)$ as in \eqref{t map},
\begin{equation}
  \label{eq:update-mix FV}\nonumber
\hat w_{\nn}\propto
w_{\mm}\,\mathrm{PU}_{\alpha}(\yy_{m+1:m+n}\mid \yy_{1:m})
   \quad   \textrm{for}\quad  \nn = t(\yy,\mm) \,,\sum_{\nn
      \in t(\yy,M)} \hw_\nn =1\,,
\end{equation}
and
\begin{equation}\label{eq: FV algorithm propagation}
\psi_{t}\Bigg(\sum_{\mm\in M}w_{\mm}\Pi_{\alpha+\sum_{i=1}^{K_{m}}m_{i}\delta_{y_{i}^{*}}}\Bigg)
=
\sum_{\nn\in L(M)}p(M,\nn,t)
\Pi_{\alpha+\sum_{i=1}^{K_{m}}n_{i}\delta_{y_{i}^{*}}},
\end{equation}
with
\begin{equation}\label{FV propagation mix weight}
p(M,\nn,t)=\sum_{\mm\in M,\, \mm\ge \nn}w_{\mm}p_{\mm,\nn}(t)
\end{equation} 
and $p_{\mm,\nn}(t)$ as in \eqref{transition probabilities}.
\end{proposition}

Note that  the update operation \eqref{eq: FV algorithm update} preserves the number of components in the mixture, while the prediction operation \eqref{FV propagation} increases this number. The intuition behind this point is analogous to the illustration in Section \ref{sec: illustration}, where the prior (node $(0,0)$) is update to the posterior (node $(2,1)$) and propagated into a mixture (coloured nodes), with the obvious difference that the recorded number of distinct values is unbounded.

Algorithm \ref{alg: FV} describes in pseudo-code the implementation of
the filter for FV processes. 

\IncMargin{5mm}
\begin{algorithm}[t]
\footnotesize
\vspace{2mm}
\hspace{-5mm} \KwData{$\yy_{t_{j}}=(y_{t_{j},1},\ldots,y_{t_{j},m_{t_{j}}})$ at times $t_{j}$, $j=0,\ldots,J$, as in \eqref{DP data}}
\hspace{-5mm} Set prior parameters $\alpha=\theta P_{0}$, $\theta>0$, $P_{0}\in \MM_{1}(\Y)$\\[0mm]

 \SetKwBlock{Begin}{Initialise}{end}
 \Begin{
$\yy\leftarrow\emptyset$, $\yy^{*}=\emptyset$, $m\leftarrow0$, $\mm\leftarrow\oo$, $M\leftarrow\{\oo\}$, $K_{m}\leftarrow0$, 
$w_{\oo}\leftarrow1$\\
}

 \SetKwBlock{Begin}{For $j=0,\ldots,J$}{end}
 \Begin{
 \SetKwBlock{Begin}{Compute data summaries}{end}
 \Begin{
read data $\yy_{t_{j}}$\\
$m\leftarrow m+\text{card}(\yy_{t_{j}})$\\
$\yy^{*}\leftarrow$ distinct values in $\yy^{*}\cap \yy_{t_{j}}$\\
$K_{m}=\text{card}(\yy^{*})$
}

 \SetKwBlock{Begin}{Update operation}{end}
 \Begin{
 
 \SetKwBlock{Begin}{for $\mm\in M$}{end}
  \Begin{
$\nn\leftarrow t(\yy_{t_{j}},\mm)$\\
$w_{\nn}\leftarrow
w_{\mm}\,\mathrm{PU}_{\alpha}(\yy_{t_{j}}\mid \yy)
$\\[0mm]
   }
    $M\leftarrow t(\yy_{t_{j}},M)$\\
 \SetKwBlock{Begin}{for $\mm\in M$}{end}
  \Begin{       $w_{\mm}\leftarrow w_{\mm}/\sum_{\ell \in M}w_{\ell}$
  }
    $X_{t_{j}}\mid \yy,\yy_{t_{j}}\sim    \sum_{\mm\in M}w_{\mm}\Pi_{\alpha+\sum_{i=1}^{K_{m}}m_{i}\delta_{y_{i}^{*}}}$
 }
 
    \SetKwBlock{Begin}{Propagation operation}{end}
 \Begin{
    \SetKwBlock{Begin}{for $\nn\in L(M)$}{end}
  \Begin{
$w_{\nn}\leftarrow p(M,\nn,t_{j+1}-t_{j})$ as in \eqref{FV propagation mix weight}\\
   }
   $M\leftarrow L(M)$\\
   $X_{t_{j+1}}\mid \yy,\yy_{t_{j}}\sim    \sum_{\mm\in M} w_{\mm}\Pi_{\alpha+\sum_{i=1}^{K_{m}}m_{i}\delta_{y_{i}^{*}}}$\\
  }

$\yy\leftarrow\yy\cup \yy_{t_{j}}$
   }
 \caption{Filtering algorithm for FV signals}\label{alg: FV}
\end{algorithm}


\subsection{Filtering Dawson--Watanabe signals}\label{subsec: DW-DW}

Let now the signal $Z_{t}$ follow a DW process with transition
function \eqref{DW trans-funct}, with invariant measure given by the law $\Gamma^{\beta}_{\alpha}$ of a gamma random measure; see \eqref{gamma process series}. 
We assume that, given the signal state, observations are drawn from a Poisson point process with intensity $z$, i.e., as in \eqref{poisson data} with $Z_{t}=z$. Analogously to the FV case, since $z$ is almost surely discrete, a sample $\yy_{1:m}=(y_{1},\ldots,y_{m})$ from \eqref{poisson data} will feature $K_{m}\le m$ ties among the observations with positive probability. To this end, we adopt the same notation as in Section \ref{sec: FV results}. 

The following Lemma states in our notation the special case of the conjugacy for mixtures of gamma random measures which is of interest here; see Section \ref{subsec: DW-static}.

\begin{lemma}\label{prop: DW update}
Let $\M$ be as in \eqref{set of multiplicities}, $M\subset \M$, $\alpha$ as in \eqref{alpha measure} and $z$ be the mixture of gamma random measures
\begin{equation*}
z\sim \sum_{\mm \in M}w_{\mm}\Gamma^{\beta+1}_{\alpha+\sum_{i=1}^{K_{m}}m_{i}\delta_{y_{i}^{*}}},
\end{equation*}
with $\sum_{\mm\in M}w_\mm=1$. Given an additional $n$-sized sample $\yy_{m+1:m+n}$ from $z$ as in \eqref{poisson data}  with multiplicities $\nn$, the update operator \eqref{update operator} yields
\begin{align}\label{eq: NP-update}
\phi_{\yy_{m:m+n}}\Bigg(\sum_{\mm \in M}w_{\mm}\Gamma^{\beta+1}_{\alpha+\sum_{i=1}^{K_{m}}m_{i}\delta_{y_{i}^{*}}}\Bigg)
=
\sum_{\mm\in M}\hat w_{\mm}
\Gamma^{\beta+2}_{\alpha+\sum_{i=1}^{K_{m+n}}(m_{i}+n_{i})\delta_{y_{i}^{*}}},
\end{align}
with $\hat w_{\mm}$ as in \eqref{updated mixture weights}.
\end{lemma}

The updated distribution is thus still a mixture of gamma random measures with updated parameters and the same number of components.

The following Theorem formalises our main result on DW processes, showing that the family of finite mixtures of gamma random measures is conjugate with respect to data as in \eqref{poisson data} with $Z_{t}=z$.  

\begin{theorem}\label{thm: DW propagation}
Let $\psi_{t}$ be the prediction operator \eqref{prediction operator} associated to a DW process with transition operator \eqref{DW trans-funct}. Let also $L(M)$ be as in \eqref{L(M)}. Then the prediction operator yields as $t$-time-ahead propagation the finite mixture of gamma random measures
\begin{equation}\label{gamma propagation}
\psi_{t}\Big(\Gamma^{\beta+\ss}_{\alpha+\sum_{i=1}^{K_{m}}m_{i}\delta_{y_{i}^{*}}}\Big)=
\sum_{\nn\in L(\mm)}
\tilde p_{\mm,\nn}(t)
\Gamma^{\beta+\SS_{t}}_{\alpha+\sum_{i=1}^{K_{m}}n_{i}\delta_{y_{i}^{*}}},
\end{equation}
where
\begin{equation}\label{tilde pmn}
\tilde p_{\mm,\nn}(t)=\mathrm{Bin}(\norm{\mm}-|\nn|;\,\norm{\mm},p(t))p(\nn;\, \mm,|\nn|),
\end{equation}
and
\begin{equation}\label{bernoulli parameter}
p(t)=\SS_{t}/\SS_{0},\quad 
\SS_{t}=\frac{\beta \SS_{0}}{(\beta+\SS_{0})e^{\beta t/2}-\SS_{0}}, \quad \SS_{0}=\ss.
\end{equation}
with $p(\nn;\,\mm,|\nn|)$ as in \eqref{hypergeometric} and $\mathrm{Bin}(\norm{\mm}-|\nn|;\,\norm{\mm},p(t))$ denoting a Binomial pmf with parameters $(\norm{\mm},p(t))$ evaluated at $\norm{\mm}-|\nn|$.
\end{theorem}

The transition operator of the DW process thus maps a gamma random measure into a finite mixture of gamma random measures. The time-varying mixing weights factorise into the binomial transition probabilities of a one-dimensional death process starting at the total size of previous data $\norm{\mm}$  and into a hypergeometric pmf. The intuition is that the death process regulates how many levels down the $K_{m}$ dimensional lattice are taken, and the hypergeometric probability chooses which admissible path down the graph is chosen given the arrival level. In  Figure \ref{fig: graph 1} we would have $K_{m}=2$ distinct values with multiplicites $\mm=(2,1)$ and total size $\norm{\mm}=3$. Then, e.g., $\tilde p_{(2,1),(1,1)}(t)$, is given by the probability $\mathrm{Bin}(1;\,3,p(t))$ that the death process jumps down one level from $3$ in time $t$ (Figure \ref{fig: graph 1}-(a)), times the probability $p((1,1);\, (2,1),2)$, conditional on going down one level, of reaching $(1,1)$ from $(2,1)$ instead of $(2,0)$, i.e.~of removing one item from the pair and not the singleton observation. The Binomial transition of the one-dimensional death process is subordinated to a deterministic process $\SS_{t}$ which modulates the sample size continuously in \eqref{gamma propagation}, starts at the value $\SS_{0}=s$ (cf.~the left hand side of \eqref{gamma propagation}) and converges to $0$ as $t\rightarrow \infty$. 

The result is obtained by means of a similar argument to that used for Theorem \eqref{thm: FV propagation}, jointly with the relation \eqref{gamma-dirichlet algebra} (which here suffices to apply at the margin of the process). In particular, we exploit the fact that the projection of a DW process onto an arbitrary partition of the space yields a vector of independent CIR processes.  See Section \ref{sec: filtering DW} for a proof. Analogously to the FV case, the result shows that under the present assumptions, the series expansion for the transition function \eqref{DW trans-funct} reduces to a finite sum. 

The following Proposition formalises the recursive algorithm that  evaluates the marginal posterior laws $\L(X_{t_n} | Y_{1:n})$ of a partially observed DW process, allowing to perform sequential Bayesian inference on a hidden signal of DW type by means of a finite computation and within the family of finite mixtures of gamma random measures.
Define such family as
\begin{align*}
\Fgamma =
\Bigg\{&\,\sum_{\mm\in M}w_{\mm}\Gamma^{\beta+\ss}_{\alpha+\sum_{i=1}^{K_{m}}m_{i}\delta_{y_{i}^{*}}}:\\ 
&\,s>0,\, M \subset \M,\ |M|<\infty,\, w_{\mm}\ge0,\  \sum_{\mm\in M}w_\mm=1\Bigg\},
\end{align*}
with $\M$ as in \eqref{set of multiplicities}.

\begin{proposition}\label{prop: algorithm DW}
Let $Z_{t}$ be a DW process with transition function \eqref{DW trans-funct} and invariant law $\Gamma_{\alpha}^{\beta}$ defined as in Section \ref{subsec: DW-static}, and suppose data are collected as in \eqref{poisson data} with $Z_{t}=z$. Then $\Fgamma$ is closed under the application of the update and prediction operators \eqref{update operator} and \eqref{prediction operator}. Specifically,
\begin{equation}\label{DW algorithm update}
\phi_{\yy_{m+1:m+n}}\Bigg(\sum_{\mm \in M}w_{\mm}\Gamma^{\beta+\ss}_{\alpha+\sum_{i=1}^{K_{m}}m_{i}\delta_{y_{i}^{*}}}\Bigg)
=
\sum_{\nn\in t(\yy_{m+1:m+n},M)}\hat w_{\nn}
\Gamma^{\beta+\ss+1}_{\alpha+\sum_{i=1}^{K_{m+n}}n_{i}\delta_{y_{i}^{*}}},
\end{equation}
with $t(\yy,M)$ as in \eqref{t map}, $\hat w_{\nn}$ as in Proposition \ref{prop: algorithm FV}, 
and
\begin{align}\label{DW algorithm prediction}
\psi_{t}\bigg(\sum_{\mm\in M}w_{\mm}\Gamma^{\beta+\ss}_{\alpha+\sum_{i=1}^{K_{m}}m_{i}\delta_{y_{i}^{*}}}\bigg)
=
\sum_{\nn\in L(M)}
p(M,\nn,t)
\Gamma^{\beta+\SS_{t}}_{\alpha+\sum_{i=1}^{K_{m}}n_{i}\delta_{y_{i}^{*}}}.
\end{align}
with
\begin{equation}\label{DW propagation mix weight}
p(M,\nn,t)=\sum_{\mm\in M,\, \mm\ge \nn}w_{\mm}\tilde p_{\mm,\nn}(t)
\end{equation} 
and $\tilde p_{\mm,\nn}(t)$ as in \eqref{tilde pmn} and $\SS_{t}$ as in \eqref{bernoulli parameter}.
\end{proposition}

Algorithm \ref{alg: DW} describes in pseudo-code the implementation of the filter for DW processes.

\IncMargin{0mm}
\begin{algorithm}[t]
\footnotesize
\vspace{2mm}
\hspace{-5mm} \KwData{$(m_{t_{j}},\yy_{t_{j}})=(m_{t_{j}},y_{t_{j},1},\ldots,y_{t_{j},m_{t_{j}}})$ at times $t_{j}$, $j=0,\ldots,J$, as in \eqref{poisson data}}
\hspace{-5mm} Set prior parameters $\alpha=\theta P_{0}$, $\theta>0$, $P_{0}\in \MM_{1}(\Y)$, $\beta>0$\\[0mm]

 \SetKwBlock{Begin}{Initialise}{end}
 \Begin{
$\yy\leftarrow\emptyset$, $\yy^{*}=\emptyset$, $m\leftarrow0$, $\mm\leftarrow\oo$, $M\leftarrow\{\oo\}$, $K_{m}\leftarrow0$, 
$w_{\oo}\leftarrow1$, $s=0$\\
}

 \SetKwBlock{Begin}{For $j=0,\ldots,J$}{end}
 \Begin{
 \SetKwBlock{Begin}{Compute data summaries}{end}
 \Begin{
read data $\yy_{t_{j}}$\\
$m\leftarrow m+\text{card}(\yy_{t_{j}})$\\
$\yy^{*}\leftarrow$ distinct values in $\yy^{*}\cap \yy_{t_{j}}$\\
$K_{m}=\text{card}(\yy^{*})$
}

 \SetKwBlock{Begin}{Update operation}{end}
 \Begin{
 
 \SetKwBlock{Begin}{for $\mm\in M$}{end}
  \Begin{
$\nn\leftarrow t(\yy_{t_{j}},\mm)$\\
$w_{\nn}\leftarrow
w_{\mm}\,\mathrm{PU}_{\alpha}(\yy_{t_{j}}\mid \yy)
$\\[0mm]
   }
    $M\leftarrow t(\yy_{t_{j}},M)$\\
 \SetKwBlock{Begin}{for $\mm\in M$}{end}
  \Begin{       $w_{\mm}\leftarrow w_{\mm}/\sum_{\ell \in M}w_{\ell}$
  }
    $X_{t_{j}}\mid \yy,\yy_{t_{j}}\sim    \sum_{\mm\in M}w_{\mm}\Gamma^{\beta+s}_{\alpha+\sum_{i=1}^{K_{m}}m_{i}\delta_{y_{i}^{*}}}$
 }

    \SetKwBlock{Begin}{Propagation operation}{end}
 \Begin{
 \SetKwBlock{Begin}{for $\nn \in L(M)$}{end}
  \Begin{
$w_{\nn}\leftarrow p(M,\nn,t_{j+1}-t_{j})$ as in \eqref{DW propagation mix weight}\\
   }
                  $M\leftarrow L(M)$\\
    $s'\leftarrow\SS_{t_{j+1}-t_{j}}$ as in \eqref{bernoulli parameter}, $\SS_{0}=s$\\
   $X_{t_{j+1}}\mid \yy,\yy_{t_{j}}\sim \sum_{\mm\in M}w_{\mm}\Gamma^{\beta+s'}_{\alpha+\sum_{i=1}^{K_{m}}m_{i}\delta_{y_{i}^{*}}}$\\
$   s\leftarrow s'$\\
  }
$\yy\leftarrow\yy\cup \yy_{t_{j}}$
   }
 \caption{Filtering algorithm for DW signals}\label{alg: DW}
\end{algorithm}


\section{Theory for computable filtering of FV and DW signals}\label{sec: computable filtering}


\subsection{Computable  filtering and duality}\label{sec: filtering and duality}

A filter is said to be computable if the sequence of filtering
distributions (the marginal laws of the signal given past and current data) can be characterised by a set of parameters whose
computation is achieved at a cost that grows at most polynomially with
the number of observations. See, e.g., \cite{CG06}.  Special cases of this framework are finite dimensional filters for which the computational cost is linear in the number of observations, the Kalman filter for linear Gaussian HMMs being the
reference model in this setting. 

Let $\X$ denote the state space of the HMM.
\cite{PR14} showed that the existence of a computable filter can be established if the following structures are embedded in the model:
\begin{description}
\item \emph{Conjugacy}: there exists a function $h(x,\mm,\theta)\geq 0$,
  where $x \in \X$,  $\mm \in \Z^{K}$ for some
  $K\in\N$, and $\theta \in \R^l$ for some $l\in\N$, and functions $t_{1}(y,\mm)$
  and $t_{2}(y,\theta)$ such that $\int h(x,\mm,\theta) \pi(dx)  =1$, for all $\mm$ and $\theta$, and
\[
\phi_y (h(x,\mm,\theta) \pi(dx) ) = h(x,t_{1}(y,\mm),t_{2}(y,\theta))\pi(dx).
\]
Here $h(x,\mm,\theta)\pi(dx)$ identifies a parametric family of distributions which is closed under Bayesian updating with respect to the observation model. 
Two types of parameters are considered, a multi-index
$\mm$ and a vector of real-valued parameters $\theta$. The update
operator $\phi_y$ maps the distribution $h(x,\mm,\theta)\pi(\d x)$,
conditional on the new observation $y$, into a density of the same
family with updated parameters $t_{1}(y,\mm)$ and $t_{2}(y,\theta)$.  
Typically $\pi(dx)$ is the prior and $h(x,\mm,\theta)$ is the Radon--Nikodym derivative of the posterior with respect to the prior, when the model is dominated. See, e.g., \eqref{WF duality function} below for an example in the Dirichlet case. 
\item[]
\item \emph{Duality}: there exists a two-component Markov
process $(M_t,\Theta_t)$ with state-space $\Z^{K} \times \R^l$ and
infinitesimal generator
 \begin{equation*}
\label{eq:gen-dual}
 \begin{split}
  (A g)(\mm,\theta) =&\, \lambda(\norm{\mm}) \rho(\theta) \sum_{i=1}^K m_i
   [g(\mm\!-\!\ee_i,\theta)\!-\!g(\mm,\theta)]\! +\! \sum_{i=1}^{l}\r_{i}(\theta)   \frac{\partial g(\mm,\theta)}{\partial \theta}
    \end{split}
\end{equation*}
acting on bounded functions, such that $(M_t,\Theta_t)$ is \emph{dual} to $X_t$ with respect to the function $h$, i.e., it satisfies
\begin{equation}
  \label{duality identity}
   \begin{split}
  \E^x[h(X_t,\mm,\theta)] = \E^{(\mm,\theta)}[h(x,M_t,\Theta_t)],
   \end{split}
\end{equation}
for all $x \in \X, \mm \in \Z^K, \theta \in \R^l, t\geq 0$.
Here $M_t$ is a death process on $\Z^{K}$, i.e.~a non-increasing pure-jump continuous time Markov process, which jumps from $\mm$ to $\mm-\ee_{i}$ at rate $\lambda(\norm{\mm})m_{i} \rho(\theta)$ and is eventually absorbed at the origin; $\Theta_{t}$ is a deterministic process assumed to evolve autonomously according to a system of ordinary differential equations $r(\Theta_t)=d\Theta_t/dt$ for some initial condition $\Theta_0=\theta_0$ and a suitable function $r:\R^l\to\R^l$ (whose $i$-coordinate is denoted by $r_i$ in the generator $A$ above and modulates the death rates of $M_t$ through $\rho(\theta)$; the expectations on the left and right hand sides are taken with respect to the law of $X_{t}$ and $(M_t,\Theta_{t})$ respectively, conditional on the respective starting points.
\end{description}

The duality condition \eqref{duality identity} hides a specific distributional relationship between the signal process $X_{t}$, which can be thought of as the forward process, and the dual process
$(M_{t},\Theta_{t})$, which can be thought of as unveiling some features of the time reversal structure of $X_{t}$. Informally, the death process can be considered as the time reversal of collecting data points if they come at random times, and the deterministic process, in the CIR example (see Section \eqref{sec: illustration}), can be considered as a continuous reversal of increasing (by steps) the sample size. 
More formally, in the well known duality relation between WF
processes and Kingman's coalescent, for example, the latter describes
the genealogical tree back in time of a sample of individuals from the current population. See \cite{GS10}. See also \cite{JK14} for a review of duality structures for Markov processes. 

Note that a local sufficient condition for \eqref{duality identity} is
\begin{equation}\label{local duality}
(\A h(\cdot,\mm,\theta))(x)
=(A h(x,\cdot,\cdot))(\mm,\theta), 
\end{equation}
for all $\forall x\in \X, \mm\in \Z^K, \theta \in \R^l$, where $A$ is as above and $\A$ denotes the generator of the signal $X_t$, which is usually easier to check.

Under the above conditions, Proposition 2.3 of \cite{PR14} shows that given the family of distributions
\begin{equation*}
\F=\Big\{h(x,\mm,\theta)\pi(\d x),\, \mm \in \Z^K, \theta \in \R^l\Big\},
\end{equation*} 
if $\nu \in \F$, then the filtering distribution $\nu_n$ which satisfies \eqref{recursion} is a finite mixture of distributions in $\F$ with parameters that can be computed recursively. This in turn implies that the family of finite mixtures of elements of $\F$ is closed under the iteration of update and prediction operations. 

The interpretation is along the lines of the illustration of Section \ref{sec: illustration}. Here $\pi$, the
stationary measure of the forward process, plays the role of the prior distribution and is represented by the
origin of $\Z^{K}$ (see Figure \ref{fig: graph 1}), which encodes the lack of information on the data generating distribution. 
Given a sample from the conjugate observation model, a single component posterior distribution is identified by a node different from the origin in $\Z^{K}$. The propagation operator then gives positive mass at all nodes which lie beneath the current nodes with positive mass. The filtering distribution then evolves within the family of finite mixtures of elements of $\F$.

\subsection{Computable filtering for Fleming--Viot processes}\label{sec: filtering FV}

In the present and the following Section we adopt the same notation  used in Section \ref{sec: main results}. 
We start by formally stating the precise form for the transition probabilities of the death processes involved in the FV filtering. Here the key point to observe is that since the number of distinct types observed in the discrete samples from a FV process is $K_{m}\le m$, we only need to consider a generic death processes on $\Z^{K_{m}}$ and not on $\Z^{\infty}$. For FV processes, the deterministic component $\Theta_t$ is constant: here we set $\Theta_t =1$ for every $t$ and we omit $\theta$ from the arguments of the duality function $h$.\\
The following Lemma will provide the building block for the proof of Theorem \ref{thm: FV propagation}. In particular, it shows that the transition probabilities of the dual death  process are of the form required as coefficients in the expansion \eqref{transition probabilities}.

\begin{lemma}\label{lemma: death transitions probabilities}
Let $M_{t}\subset\Z^{\infty}$ be a death process that starts from $M_{0}=\mm_{0}\in\M$, $\M$ as in \eqref{set of multiplicities}, and jumps from $\mm$ to $\mm-\ee_{i}$ at rate $m_{i}(\theta+\mm-1)/2$, with generator
\begin{equation*}
\frac{\theta+\norm{\mm}-1}{2}\sum_{i\ge1}m_{i}h(\xx,\mm-\ee_{i})-\frac{\norm{\mm}(\theta+\norm{\mm}-1)}{2}h(\xx,\mm).
\end{equation*}
Then the transition probabilities for $M_{t}$ are
 \begin{equation}\label{M_transition probabilities}
 p_{\mm,\mm-\ii }(t)=
\left\{
\begin{array}{ll}
e^{-\lambda_{\norm{\mm}}t},& \ii=\oo,\\
  C_{\norm{\mm} ,\norm{\mm} -|\ii|}(t) p(\ii;\,\mm,|\ii|),
 \quad \quad &  \oo< \ii\le \mm,
\end{array}
\right. 
 \end{equation}
where 
 \begin{equation*}
C_{\norm{\mm} ,\norm{\mm} -|\ii|}(t)=
\bigg(\prod_{h=0}^{|\ii|-1}\lambda_{\norm{\mm} -h}\bigg)
 (-1)^{|\ii|}\sum_{k=0}^{|\ii|}\frac{e^{-\lambda_{\norm{\mm} -k}t}}{\prod_{0\le h\le |\ii|,h\ne k}(\lambda_{\norm{\mm} -k}-\lambda_{\norm{\mm} -h})},
\end{equation*} 
$\lambda_{n}=n(\theta+n-1)/2$ and $ p(\ii;\,\mm,|\ii|)$ as in \eqref{hypergeometric}, and 0 otherwise.
\end{lemma}
\begin{proof}
Since $|\mm_{0}|<\infty$, for any such $\mm_{0}$ the proof is analogous to that of Proposition 2.1 in \cite{PR14}.
\end{proof}

The following Proof of the conjugacy for mixtures of Dirichlet processes is due to \cite{A74} and outlined here for the ease of the reader. 

\bigskip
\noindent \textbf{Proof of Lemma \ref{lem: FV update multi}}\\
\noindent The distribution $x$ is a mixture of Dirichlet processes with mixing measure $H(\cdot)=\sum_{\mm\in M}w_{\mm}\delta_{\mm}(\cdot)$ on $M$ and transition measure
\begin{equation*}
\alpha_{\mm}(\cdot)
=\alpha(\cdot)+\sum_{j=1}^{K_{m}}m_{j}\delta_{y_{j}^{*}}(\cdot)
=\alpha(\cdot)+\sum_{i=1}^{m}\delta_{y_{i}}(\cdot),
\end{equation*}
where $\yy_{1:m}$ is the full sample.
See Section \ref{subsec: FV-static}. Lemma 1 and Corollary 3.2' in \cite{A74} now imply that
\begin{equation*}
x\mid \mm,\yy_{m+1:m+n}\sim
\Pi_{\alpha_{\mm}(\cdot)+\sum_{i=m+1}^{n}\delta_{y_{i}}(\cdot)}
=\Pi_{\alpha(\cdot)+\sum_{i=1}^{n}\delta_{y_{i}}}
\end{equation*}
and $H(\mm\mid \yy_{m+1:m+n})\propto w_{\mm}\,\mathrm{PU}_{\alpha}(\yy_{m+1:m+n}\mid \yy_{1:m})$.\qed

\bigskip
As preparatory for the main result on FV processes, we derive here in detail the propagation step for WF processes, which is due to \cite{PR14}. Let 
\begin{equation}\label{eq:WF-generator}
\Gwf f(\xx)=
\frac{1}{2}\sum_{i,j=1}^{K}x_{i}(\delta_{ij}-x_{j})\frac{\partial^{2}f(\xx)}{\partial x_{i}\partial x_{j}}
+\frac{1}{2}\sum_{i=1}^{K}(\alpha_{i}-\na x_{i})\frac{\partial f(\xx)}{\partial x_{i}}
\end{equation}
be the infinitesimal generator of a $K$-dimensional WF diffusion, with $\alpha_{i}>0$ and $\sum_{i}\alpha_{i}=\theta$. 
Here $\delta_{ij}$ denotes Kronecker delta and $\Gwf$ acts on
$C^{2}(\DK)$ functions, with 
\begin{equation}\label{eq:K-simplex}\notag
\DK=\Big\{x\in[0,1]^{K}:\, \sum\nolimits_{i=1}^{K}x_{i}=1\Big\}.
\end{equation}

\begin{proposition}\label{prop: PR14-WF-infty}
Let $\XX_{t}$ be a WF diffusion with generator \eqref{eq:WF-generator}
and Dirichlet invariant measure on \eqref{eq:K-simplex} denoted $\pi_{\aa}$. Then, for any $\mm\in\Z^{\infty}$ such that $|\mm|<\infty$,
\begin{equation}\label{nu-m propagation inf}
\psi_{t}\big(\pi_{\aa+\mm}\big)
=\sum_{\substack{\oo \le \ii\le \mm}}p_{\mm,\mm-\ii}(t)\pi_{\aa+\mm-\ii},
\end{equation}
with $p_{\mm,\mm-\ii}(t)$ as in \eqref{lemma: death transitions probabilities}.
\end{proposition}
\begin{proof}
Define
\begin{equation}\label{WF duality function}
h(\xx,\mm)=\frac{\Gamma(\na+\norm{\mm})}{\Gamma(\na)}\prod_{i=1}^{K}\frac{\Gamma(\alpha_{i})}{\Gamma(\alpha_{i}+m_{i})}\xx^{\mm},
\end{equation}
which is in the domain of $\Gwf$. A direct computation shows that
\begin{equation}\label{eq:WF-gen-xm}
\begin{split}\notag
\Gwf h(\xx,\mm)
=&\,
\sum_{i=1}^{K}\bigg(\frac{\alpha_{i}m_{i}}{2}+\binom{m_{i}}{2}\bigg)\frac{\Gamma(\theta+|\mm|)}{\Gamma(\theta)}\prod_{j=1}^{K}\frac{\Gamma(\alpha_{j})}{\Gamma(\alpha_{j}+m_{j})}\xx^{\mm-\ee_{i}}\\
&\,-\sum_{i=1}^{K}\bigg(\frac{\theta m_{i}}{2}+\binom{m_{i}}{2}+\frac{1}{2}m_{i}\sum_{j\ne i}m_{j}\bigg)\frac{\Gamma(\theta+|\mm|)}{\Gamma(\theta)}\prod_{j=1}^{K}\frac{\Gamma(\alpha_{j})}{\Gamma(\alpha_{j}+m_{j})}\xx^{\mm}\\
=&\,\frac{\theta+|\mm|-1}{2}\sum_{i=1}^{K}m_{i}h(\xx,\mm-\ee_{i})-\frac{|\mm|(\theta+|\mm|-1)}{2}h(\xx,\mm).
\end{split}
\end{equation}
Hence, by \eqref{local duality}, the death process $M_{t}$ on $\Z^{K}$, which jumps from $\mm$ to $\mm-\ee_{i}$ at rate $m_{i}(\theta+|\mm|-1)/2$, is dual to the WF diffusion with generator $\Gwf$ with respect to \eqref{WF duality function}.
From the definition \eqref{prediction operator} of the prediction operator now we have
\begin{align*}
\psi_{t}\big(\pi_{\aa+\mm}\big)(\d\xx')
=&\,\int_{\X}h(\xx,\mm)\pi_{\aa}(\d\xx)P_{t}(\xx,\d \xx')\\
=&\,\int_{\X}h(\xx,\mm)\pi_{\aa}(\d\xx')P_{t}(\xx',\d \xx)\\
=&\,\pi_{\aa}(\d \xx')\E^{\xx'}[h(\XX_{t},\mm)]\\
=&\,\pi_{\aa}(\d \xx')\E^{\mm}[h(\xx',M_{t})]\\
=&\,\pi_{\aa}(\d \xx')\sum_{\substack{\oo \le \ii\le \mm}}p_{\mm,\mm-\ii}(t)h(\xx',\mm-\ii)\\
=&\,\sum_{\substack{\oo \le \ii\le \mm}}p_{\mm,\mm-\ii}(t)\pi_{\aa+\mm-\ii}(\d\xx')
\end{align*}
where the second equality holds in virtue of the reversibility of
$\XX_{t}$ with respect to $\pi_{\aa}$, the fourth by the duality \eqref{duality identity} established above together with \eqref{M_transition probabilities} and the fifth from Lemma \ref{lemma: death transitions probabilities}.
\end{proof}

\bigskip

The following proves the propagation step for FV processes by making use of the previous result and by exploiting the strategy outlined in Figure \ref{fig: scheme}.

\bigskip
\noindent\textbf{Proof of Theorem \ref{thm: FV propagation}}\\
\noindent Fix an arbitrary partition $(A_{1},\ldots,A_{K})$ of $\Y$ with $K$ classes, and denote by $\tilde\mm$ the multiplicities resulting from binning $\yy_{1:m}$ into the corresponding cells. Then
\begin{equation}\label{DP to dir}
\Pi_{\alpha+\sum_{i=1}^{K_{m}}m_{i}\delta_{y_{i}^{*}}}(A_{1},\ldots,A_{K})
\sim
\pi_{\alpha+\tilde\mm},
\end{equation}
where $\Pi_{\alpha+\sum_{i=1}^{K_{m}}m_{i}\delta_{y_{i}^{*}}}(A_{1},\ldots,A_{K})$ denotes the law $\Pi_{\alpha+\sum_{i=1}^{K_{m}}m_{i}\delta_{y_{i}^{*}}}(\cdot)$ evaluated on $(A_{1},\ldots,A_{K})$.
Since the projection onto the same partition of the FV process is a $K$-dimensional WF process (see Section \ref{subsec: FV-model}), from Proposition \ref{prop: PR14-WF-infty} we have
\begin{align}\label{WF tilde propagation}\nonumber
\psi_{t}\Big(&\,\Pi_{\alpha+\sum_{i=1}^{K_{m}}m_{i}\delta_{y_{i}^{*}}}(A_{1},\ldots,A_{K})\Big)
=\psi_{t}(\pi_{\alpha+\tilde\mm})
=\sum_{\substack{\nn\in L(\tilde\mm)}}p_{\tilde\mm,\nn}(t)\pi_{\alpha+\nn}.\notag
\end{align}
Furthermore, since a Dirichlet process is characterised by its finite-dimensional projections, now it suffices to show that
\begin{equation*}
\sum_{\nn\in L(\mm)}p_{\mm,\nn}(t)\Pi_{\alpha+\sum_{i=1}^{K_{m}}n_{i}\delta_{y_{i}^{*}}}(A_{1},\ldots,A_{K})
=\sum_{\substack{\nn\in L(\tilde\mm)}}p_{\tilde\mm,\nn}(t)\pi_{\alpha+\nn}
\end{equation*}
so that the operations of propagation and projection commute. Given \eqref{DP to dir}, we only need to show that the mixture weights are consistent with respect to fragmentation and merging of classes, that is
\begin{equation}\notag
\sum_{\ii\in L(\mm):\, \tilde \ii=\nn}p_{\mm,\ii}(t)=p_{\tilde\mm,\nn}(t),
\end{equation}
where $\tilde\ii$ denotes the projection of $\ii$ onto $(A_{1},\ldots,A_{K})$.
Using \eqref{transition probabilities}, the previous in turn reduces to
\begin{equation}\label{hypergeometric marginal}\nonumber
\sum_{\ii\in L(\mm):\, \tilde \ii=\nn}p(\ii;\,\mm,m-i)=p(\nn;\,\tilde\mm,m-n),
\end{equation}
which holds by the marginalization properties of the multivariate hypergeometric distribution. Cf.~\cite{JKB97}, equation 39.3.\qed

\bigskip
The last needed result to obtain the recursive representation of Proposition \ref{prop: algorithm FV} reduces now to a simple sum rearrangement. 

\bigskip
\noindent\textbf{Proof of Proposition \ref{prop: algorithm FV}}

\noindent The update operation \eqref{eq: FV algorithm update} follows directly from Lemma \ref{lem: FV update multi}.
The prediction operation \eqref{eq: FV algorithm propagation} for elements of $\F_{\Pi}$ follows from Theorem \ref{thm: FV propagation} together with the linearity of \eqref{prediction operator} and a rearrangement of the sums, so that
\begin{align*}
\psi_{t}\Bigg(&\,\sum_{\mm\in M}w_{\mm}\Pi_{\alpha+\sum_{i=1}^{K_{m}}m_{i}\delta_{y_{i}^{*}}}\Bigg)\\
=&\,\sum_{\mm\in M}w_{\mm}
\sum_{\substack{\nn\in L(\mm)} }
p_{\mm,\nn }(t)
\Pi_{\alpha+\sum_{i=1}^{K_{m}}n_{i}\delta_{y_{i}^{*}}}\\
=&\,\sum_{\nn\in L(M)}
\Bigg(\sum_{\mm\in M,\,\mm\ge \nn}w_{\mm}p_{\mm,\nn}(t)\Bigg)
\Pi_{\alpha+\sum_{i=1}^{K_{m}}n_{i}\delta_{y_{i}^{*}}}.
\end{align*}
\qed

\subsection{Computable filtering for Dawson--Watanabe processes}
\label{sec: filtering DW}

The following Lemma, used later, recalls the propagation step for one dimensional CIR processes.

\begin{lemma}\label{lemma: 1-CIR propagation}
Let $Z_{i,t}$ be a CIR process with generator \eqref{CIR generator} and invariant distribution $\mathrm{Ga}(\alpha_{i},\beta)$. Then
\begin{equation}\label{CIR propagation}\notag
\psi_{t}\big(\mathrm{Ga}(\alpha_{i}+m,\beta+\ss)\big)
=\sum\nolimits_{j=0}^{m}\mathrm{Bin}(m-j;\,m,p(t))
\mathrm{Ga}(\alpha_{i}+m-j,\beta+\SS_{t}),
\end{equation}
where
\begin{equation}\notag
p(t)=\SS_{t}/\SS_{0},\quad 
\SS_{t}=\frac{\beta \SS_{0}}{(\beta+\SS_{0})e^{\beta t/2}-\SS_{0}}, \quad \SS_{0}=\ss.
\end{equation}
\end{lemma}

\begin{proof}
It follows from Section 3.1 in \cite{PR14} by letting $\alpha=\delta/2$,  $\beta=\gamma/\sigma^{2}$ and $\SS_{t}=\Theta_{t}-\beta$.
\end{proof}

As preparatory for proving the  main result on DW processes, assume the signal $\ZZ_{t}=(Z_{1,t},\ldots,Z_{K,t})$ is a vector of independent CIR components $Z_{i,t}$ each with generator 
\begin{equation}\label{CIR generator}
\Gcir_{i}f(z_{i})=\frac{1}{2}(\alpha_{i}-\beta z_{i})f'(z_{i})+\frac{1}{2}z_{i}f''(z_{i}),
\end{equation}
acting on $C^{2}([0,\infty))$ functions which vanish at infinity. See \cite{KW71}.
The next proposition identifies the dual process for $\ZZ_{t}$.

\begin{theorem}\label{thm: multi CIR duality}
Let $Z_{i,t}$, $i=1,\ldots,K$, be independent CIR processes each with generator \eqref{CIR generator} parametrised by $(\alpha_i,\beta)$, respectively. For $\aa\in\R_{+}^{K}$ and $\theta=|\aa|$, define $h_{\alpha_{i}}^{C}:\R_{+}\times\Z\times\R_{+}$ as
\begin{equation}\label{CIR duality function}\notag
h_{\alpha_{i}}^{C}(z,m,\ss)=
\frac{\Gamma(\alpha_{i})}{\Gamma(\alpha_{i}+m)}
\bigg(\frac{\beta+\ss}{\beta}\bigg)^{\alpha_{i}}
(\beta+\ss)^{m}z^{m}e^{-\ss z}.
\end{equation}
Let also $h^{W}:\R_{+}^{K}\times\Z^{K}$ be as in \eqref{WF duality function} and define $h:\R_{+}^{K}\times\Z^{K}\times\R_{+}$ as
\begin{equation}\label{thm: multi CIR duality function}\notag
h(\zz,\mm,\ss)=h_\theta^{C}(|\zz|,\norm{\mm},\ss)h^{W}(\xx,\mm),
\end{equation}
where $\xx=\zz/|\zz|$.
Then the joint process $\{(Z_{1,t},\ldots,Z_{K,t}),t\ge0\}$ is dual, in the sense of \eqref{duality identity}, to the process $\{(\mathbf{M}_{t},\SS_{t}),t\ge0\}\subset\Z^{K}\times\R_{+}$ with generator
\begin{equation}\label{K-CIR dual generator}
\begin{split}
\Gn g(\mm,\ss)
  =&\, \frac{1}{2}\norm{\mm}(\beta+\ss) \sum_{i=1}^{K} \frac{m_i}{\norm{\mm}}
[   g(\mm-\ee_{i},\ss)-g(\mm,\ss)]\\
&\,-\frac{1}{2}\ss(\beta+\ss)   \frac{\partial g(\mm,\ss)}{\partial \ss}
\end{split}
\end{equation}
with respect to $h(\zz,\mm,\ss)$.
\end{theorem}
\begin{proof}
Throughout the proof, for ease of notation we will write $h_i^C$ instead of $h_{\alpha_i}^C$. Note first that for all $\mm\in \Z^{K}$ we have
\begin{equation}\label{composition duality function}
\prod_{i=1}^{K}h_i^{C}(z_{i},m_{i},\ss)
=h_{\theta}^{C}(\norm{z},\norm{\mm},\ss)h^{W}(\xx,\mm),
\end{equation}
where $x_{i}=z_{i}/\norm{z}$,
which follows from direct computation by multiplying and dividing by the correct ratios of gamma functions and by writing $\prod_{i=1}^{K}z_{i}^{m_{i}}=|z|^{m}
\prod_{i=1}^{K}x_{i}^{m_{i}}$.
We show the result for $K=2$, from which the statement for general $K$ case follows easily.
From the independence of the CIR processes, the generator $(Z_{1,t},Z_{2,t})$ applied to the left hand side of \eqref{composition duality function} is
\begin{align}\label{2CIR generator}
(\Gcir_{1}+\Gcir_{2})h^{C}_{1}h^{C}_{2}
=&\,h^{C}_{2}\Gcir_{1}h^{C}_{1}+h^{C}_{1}\Gcir_{2}h^{C}_{2}.
\end{align}
A direct computation shows that
\begin{align*}
\Gcir_{i}h^{C}_{i}
=&\,
\frac{m_{i}}{2}(\beta+\ss)h^{C}_{i}(z_{i},m_{i}-1,\ss)
+\frac{\ss}{2}(\alpha_{i}+m_{i})h^{C}_{i}(z_{i},m_{i}+1,\ss)\notag\\
&\,-\frac{1}{2}(\ss(\alpha_{i}+m_{i})+m_{i}(\beta+\ss))h^{C}_{i}(z_{i},m_{i},\ss).
\end{align*}
Substituting in the right hand side of \eqref{2CIR generator} and collecting terms with the same coefficients gives
\begin{align*}
&\frac{\beta+\ss}{2}\Big[m_{1}h^{C}_{1}(z_{1},m_{1}-1,\ss)h^{C}_{2}(z_{2},m_{2},\ss)
 +m_{2}h^{C}_{1}(z_{1},m_{1},\ss)h^{C}_{2}(z_{2},m_{2}-1,\ss)\Big]\\
&\,
+ \frac{\ss}{2}\Big[(\alpha_{1}+m_{1})h^{C}_{1}(z_{1},m_{1}+1,\ss)h^{C}_{2}(z_{2},m_{2},\ss)\\
&\,\quad \quad  +(\alpha_{2}+m_{2})h^{C}_{1}(z_{1},m_{1},\ss)h^{C}_{2}(z_{2},m_{2}+1,\ss)\Big]\\
&\,-\frac{1}{2}(\ss(\alpha+m)+m(\beta+\ss))h^{C}_{1}(z_{1},m_{1},\ss)h^{C}_{2}(z_{2},m_{2},\ss)\notag
\end{align*}
with $\alpha=\alpha_{1}+\alpha_{2}$ and $m=m_{1}+m_{2}$.
From \eqref{composition duality function} we now have
\begin{align}\notag
&\,
 \frac{\beta+\ss}{2}h_\theta^{C}(|z|,m-1,\ss)\Big[
 m_{1}h^{W}(\xx,\mm-\ee_{1},\ss)
 +m_{2}h^{W}(\xx,\mm-\ee_{2},\ss)\Big]\\
&\,
+ \frac{\ss}{2} h_\theta^{C}(|z|,m+1,\ss)\Big[(\alpha_{1}+m_{1})h^{W}(\xx,\mm+\ee_{1},\ss)\notag\\
&\,\quad \quad +(\alpha_{2}+m_{2})h^{W}(\xx,\mm+\ee_{2},\ss)\Big]\notag\\
&\,-\frac{1}{2}(\ss(\alpha+m)+m(\beta+\ss))
h_\theta^{C}(|z|,m,\ss)h^{W}(\xx,\mm,\ss)\notag.
\end{align}
Then
\begin{align}\label{gen for duality check}
\begin{split}
(\Gcir_{1}+&\,\Gcir_{2})h^{C}_{1}h^{C}_{2}\\
=&\,
\frac{\beta+\ss}{2}\Big[m_{1}h(\zz,\mm-\ee_{1},\ss)+m_{2}h(\zz,\mm-\ee_{2},\ss)\Big]\\
&\,
+ \frac{\ss}{2}\Big[(\alpha_{1}+m_{1})h(\zz,\mm+\ee_{1},\ss)+(\alpha_{2}+m_{2})h(\zz,\mm+\ee_{2},\ss)\Big]\\
&\,-\frac{1}{2}(\ss(\alpha+m)+m(\beta+\ss))
h(\zz,\mm,\ss).
\end{split}\end{align}
Noting now that
\begin{align*}
\frac{\partial}{\partial \ss}h(\zz,\mm,\ss)
=&\,\frac{\alpha+m}{\beta+\ss}
h(\zz,\mm,\ss)
-\frac{\alpha_{1}+m_{1}}{\beta+\ss}h(\zz,\mm+\ee_{1},\ss)\\
&\,-\frac{\alpha_{2}+m_{2}}{\beta+\ss}h(\zz,\mm+\ee_{2},\ss),
\end{align*}
an application of \eqref{K-CIR dual generator} on $h(\zz,\mm,\ss)$
shows that $(\Gn h(\zz,\cdot,\cdot))(\mm,\ss)$ equals the right hand side of \eqref{gen for duality check},
so that \eqref{local duality} holds, giving the result.
\end{proof}

The previous Theorem extends the gamma-type duality showed for one dimensional CIR processes in \cite{PR14}. Although the components of $\ZZ_{t}$ are independent, the result is not entirely trivial. Indeed the one-dimensional CIR process is dual to a two-components process given by a one-dimensional death process and a one-dimensional deterministic dual. The previous result shows that $K$ independent CIR processes have dual not given by a $K$ independent versions of the CIR dual, but by a death process on $\Z^{K}$ modulated by a single deterministic process.
Specifically, here the dual component $\textbf{M}_{t}$ is a $K$-dimensional death process on $ \Z^{K}$ which, conditionally on $\SS_{t}$, jumps from $\mm$ to $\mm-\ee_{i}$ at rate $2m_{i}(\beta+\SS_{t})$, and $\SS_{t}\in \R_{+}$ is a nonnegative deterministic process driven by the logistic type differential equation
\begin{equation}\label{ODE}
\frac{\d \SS_{t}}{\d t}=-\frac{1}{2}\SS_{t}(\beta+\SS_{t}).
\end{equation}
The next Proposition formalises the propagation step for multivariate CIR processes. Denote by $\textbf{Ga}(\aa,\beta)$ the product of gamma distributions $$\text{Ga}(\alpha_{1},\beta)\times\cdots\times\text{Ga}(\alpha_{K},\beta),$$ with $\aa=(\alpha_{1},\ldots,\alpha_{K})$.


\begin{proposition}\label{prop: K-CIR implications}
Let $\{(Z_{1,t},\ldots,Z_{K,t}),t\ge0\}$ be as in Theorem \ref{thm: multi CIR duality}. Then
\begin{align}\label{multi CIR propagation}
\psi_{t}\big(&\mathbf{Ga}(\aa+\mm,\beta+\ss)\big)=\\
=&\,\sum_{i=0}^{\norm{\mm}}\mathrm{Bin}(\norm{\mm}-i;\,\norm{\mm},p(t))\mathrm{Ga}(\na+\norm{\mm}-i,\beta+\SS_{t})\notag\\
&\,\times\sum_{\substack{\oo \le \ii\le \mm, |\ii|=i}}p(\ii;\, \mm,i)
\pi_{\aa+\mm-\ii},\notag
\end{align}
where $\mathrm{Bin}(\norm{\mm}-i;\,\norm{\mm},p(t))$ and $p(\ii;\, \mm,i)$ are as in \eqref{tilde pmn}.
\end{proposition}

\begin{proof}
From independence we have
\begin{align*}
\psi_{t}\big(\mathbf{Ga}(\aa+\mm,\beta+\ss)\big)
=&\,
\prod_{i=1}^{K}\psi_{t}\big(\mathrm{Ga}(\alpha_{i}+m_{i},\beta+\ss)\big).
\end{align*}
Using Lemma \ref{lemma: 1-CIR propagation} in the Appendix, the previous equals
\begin{align*}
\prod_{i=1}^{K}&\sum_{j=0}^{m_{i}}\mathrm{Bin}(m_{i}-j;\,m_{i},p(t))
\mathrm{Ga}(\alpha_{i}+m_{i}-j,\beta+\SS_{t})\\
=&\,
\sum_{i_{1}=0}^{m_{1}}\mathrm{Bin}(m_{1}-i_{1};\,m_{1},p(t))
\mathrm{Ga}(\alpha_{1}+m_{1}-i_{1},\beta+\SS_{t})\\
&\,\times\cdots\times
\sum_{i_{K}=0}^{m_{K}}\mathrm{Bin}(m_{K}-i_{K};\,m_{K},p(t))
\mathrm{Ga}(\alpha_{K}+m_{K}-i_{K},\beta+\SS_{t}).
\end{align*}
Using now the fact that a product of Binomials equals the product of a Binomial and an hypergeometric distribution, we have
\begin{equation*}
\sum_{i=0}^{\norm{\mm}}\mathrm{Bin}(\norm{\mm}-i;\,\norm{\mm},p(t))\sum_{\substack{\oo \le \ii\le \mm, |\ii|=i}}p(\ii;\, \mm,i)
\prod_{j=1}^{K}\mathrm{Ga}(\alpha_{j}+m_{j}-i_{j},\beta+\SS_{t})
\end{equation*}
which, using \eqref{gamma-dirichlet algebra}, yields \eqref{multi CIR propagation}. Furthermore, \eqref{bernoulli parameter} is obtained by solving \eqref{ODE} and by means of the following argument. The one dimensional death process that drives $|\textbf{M}_{t}|$ in Theorem \ref{thm: multi CIR duality}, jumps from $\norm{\mm}$ to $\norm{\mm}-1$ at rate $\norm{\mm}(\beta+S_{t})/2$, see \eqref{K-CIR dual generator}. The probability that $|\textbf{M}_{t}|$ remains in $\norm{\mm}$ in $[0,t]$ if it is in $\norm{\mm}$ at time 0, here denoted $P(\norm{\mm}\mid \norm{\mm},S_{t})$, is then
\begin{equation*}
P(\norm{\mm}\mid \norm{\mm},S_{t})=\exp\bigg\{-\frac{\norm{\mm}}{2}\int_{0}^{t}(\beta+S_{u})\d u\bigg\}=\left(\frac{\beta}{(\beta+\ss)e^{\beta t/2}-\ss}\right)^{\norm{\mm}}.
\end{equation*}
The probability of a jump from $\norm{\mm}$ to $\norm{\mm}-1$ occurring in $[0,t]$ is
\begin{align*}
P&\,(\norm{\mm}-1\mid \norm{\mm},S_{t})\\
=&\!\int_{0}^{t}\!
\exp\bigg\{-\frac{\norm{\mm}}{2}\int_{0}^{s}(\beta+S_{u})\d u\bigg\}
\frac{\norm{\mm}}{2}S_{s}\\
&\,\times\exp\bigg\{-\frac{\norm{\mm}-1}{2}\int_{s}^{t}(\beta+S_{u})\d u\bigg\}\d s\\
=&\,\frac{\norm{\mm}}{2}
\exp\bigg\{-\frac{\norm{\mm}}{2}\int_{0}^{t}(\beta+S_{u})\d u\bigg\}\\
&\,\times\int_{0}^{t}S_{s}
\exp\bigg\{\bigg(\frac{\norm{\mm}}{2}-\frac{\norm{\mm}-1}{2}\bigg)\int_{s}^{t}(\beta+S_{u})\d u\bigg\}\d s\\
=&\,\norm{\mm}
\exp\bigg\{-\frac{\norm{\mm}}{2}\int_{0}^{t}(\beta+S_{u})\d u\bigg\}\\
&\,\times\bigg(1-\exp\bigg\{\bigg(\frac{\norm{\mm}}{2}-\frac{\norm{\mm}-1}{2}\bigg)\int_{0}^{t}(\beta+S_{u})\d u\bigg\}\bigg)\\
=&\,\norm{\mm}
\bigg(\exp\bigg\{-\frac{\norm{\mm}}{2}\int_{0}^{t}(\beta+S_{u})\d u\bigg\}\\
&\,\quad \quad -\exp\bigg\{-\frac{\norm{\mm}-1}{2}\int_{0}^{t}(\beta+S_{u})\d u\bigg\}\bigg)\\
=&\,\norm{\mm}\left(\frac{\beta}{(\beta+\ss)e^{\beta t/2}-\ss}\right)^{\norm{\mm}-1}\left(1-\frac{\beta}{(\beta+\ss)e^{\beta t/2}-\ss}\right).
\end{align*}
Iterating the argument leads to conclude that the death process jumps from $\norm{\mm}$ to $\norm{\mm}-i$ in $[0,t]$ with probability $\text{Bin}(\norm{\mm}-i\mid \norm{\mm},p(t))$.
\end{proof}

Note that when $s\in\N$, $\text{Ga}(\alpha_{i}+m,\beta+\ss)$ is the posterior of a $\text{Ga}(\alpha_{i},\beta)$ prior given $s$ Poisson observations with total count $m$. Hence the dual component $M_{i,t}$ is interpreted as the sum of the observed values of type $i$, and $\SS_{t}\subset\R_{+}$ as a continuous version of the sample size.
In particular, \eqref{multi CIR propagation} shows that a multivariate CIR propagates a vector of gamma distributions into a mixture whose kernels factorise into a gamma and a Dirichlet distribution, and whose mixing weights are driven by a one-dimensional death process with Binomial transitions together with hypergeometric probabilities for allocating the masses.

The following Proof of the conjugacy for mixtures of gamma random measures is due to \cite{L82} and outlined here for the ease of the reader. 

\bigskip
\noindent \textbf{Proof of Lemma \ref{prop: DW update}}

%
\noindent  Since $z_{\mm}:=(z\mid \mm)\sim\Gamma^{\beta+\ss}_{\alpha+\sum_{i=1}^{K_{m}}m_{i}\delta_{y_{i}^{*}}}$, from \eqref{poisson data} we have
\begin{equation}\notag
y_{m+1},\ldots,y_{n}\mid z,\mm,n\overset{iid}{\sim}z_{\mm}/|z_{\mm}|,\quad \quad
n\mid z_{\mm}\sim\text{Po}(|z_{\mm}|).
\end{equation}
Using \eqref{gamma-dirichlet algebra} we have
\begin{align*}
\Gamma^{\beta+\ss}_{\alpha+\sum_{i=1}^{K_{m}}m_{i}\delta_{y_{i}^{*}}}
=
\text{Ga}(\theta+\norm{\mm},\beta+s)\Pi_{\alpha+\sum_{i=1}^{K_{m}}m_{i}\delta_{y_{i}^{*}}},
\end{align*}
that is $|z_{\mm}|$ and $z_{\mm}/|z_{\mm}|$ are independent with $\text{Ga}(\theta+\norm{\mm},\beta+s)$ and $\Pi_{\alpha+\sum_{i=1}^{K_{m}}m_{i}\delta_{y_{i}^{*}}}$ distribution respectively. Then we have
\begin{equation*}
z_{\mm}\mid \yy_{m+1:m+n}\sim \text{Ga}(\theta+|\nn|,\beta+s+1)\Pi_{\alpha+\sum_{i=1}^{K_{m+n}}n_{i}\delta_{y_{i}^{*}}}=
\Gamma^{\beta+\ss+1}_{\alpha+\sum_{i=1}^{K_{m+n}}n_{i}\delta_{y_{i}^{*}}}
\end{equation*}
where $\nn$ are the multiplicities of the distinct values in $\yy_{1:n}$.
Finally, by the independence of $|z_{\mm}|$ and $z_{\mm}/|z_{\mm}|$, the conditional distribution of the mixing measure follows by the same argument used in Proposition \ref{lem: FV update multi}.\qed

\bigskip
We are now ready to prove the main result for DW processes. 

\bigskip
\noindent \textbf{Proof of Theorem \ref{thm: DW propagation}}

Fix a partition $(A_{1},\ldots,A_{K})$ of $\Y$. Then by Proposition \ref{prop: K-CIR implications}
\begin{align*}
\psi_{t}\Big(&\Gamma^{\beta+\ss}_{\alpha+\sum_{i=1}^{K_{m}}m_{i}\delta_{y_{i}^{*}}}(A_{1},\ldots,A_{K})\Big)\\
=&\,\sum_{i=0}^{\norm{\mm}}\mathrm{Bin}(\norm{\mm}-i;\,\norm{\mm},p(t))\mathrm{Ga}(\na+\norm{\mm}-i,\beta+\SS_{t})\\
&\,\times\sum_{\substack{\oo \le \ii\le \tilde\mm, |\ii|=i}}p(\ii;\, \tilde\mm,i) \pi_{\aa+\tilde\mm-\ii},
\end{align*}
where $\Gamma^{\beta+\ss}_{\alpha+\sum_{i=1}^{K_{m}}m_{i}\delta_{y_{i}^{*}}}(A_{1},\ldots,A_{K})$ denotes $\Gamma^{\beta+\ss}_{\alpha+\sum_{i=1}^{K_{m}}m_{i}\delta_{y_{i}^{*}}}(\cdot)$ evaluated on $(A_{1},\linebreak\ldots,A_{K})$ and $\tilde\mm$ are the multiplicities yielded by the projection of $\mm$ onto $(A_{1},\ldots,A_{K})$.
Use now \eqref{gamma-dirichlet algebra} and \eqref{tilde pmn} to write the right hand side of \eqref{gamma propagation} as
\begin{align}\notag
\sum_{\nn\in L(\mm)}&
\tilde p_{\mm,\nn}(t)
\Gamma^{\beta+\SS_{t}}_{\alpha+\sum_{i=1}^{K_{m}}n_{i}\delta_{y_{i}^{*}}}\\
=&\,\sum_{i=0}^{\norm{\mm}}\mathrm{Bin}(\norm{\mm}-i;\,\norm{\mm},p(t))\text{Ga}(\na+\norm{\mm}-i,\beta+\SS_{t})\notag\\
&\,\times\sum_{\substack{\oo \le \nn\le \mm, |\nn|=i}}p(\nn;\, \mm,i)
\Pi_{\alpha+\sum_{j=1}^{K_{m}}(m_{j}-n_{j})\delta_{y_{j}^{*}}}.\notag
\end{align}
Since the inner sum is the only term which depends on multiplicities and  since Dirichlet processes are characterised by their finite-dimensional projections, we are only left to show that
\begin{align*}
\sum_{\substack{\oo \le \nn\le \mm, |\nn|=i}}&\,p(\nn;\, \mm,i)
\Pi_{\alpha+\sum_{j=1}^{K_{m}}(m_{j}-n_{j})\delta_{y_{j}^{*}}}(A_{1},\ldots,A_{K})\\
=&\,\sum_{\substack{\oo \le \ii\le \tilde\mm, |\ii|=i}}p(\ii;\, \tilde\mm,i) \pi_{\aa+\tilde\mm-\ii}
\end{align*}
which, in view of \eqref{DP to dir}, holds if
\begin{equation*}
\sum_{\substack{\oo \le \nn\le \mm: \tilde\nn=\ii}}p(\ii;\, \mm,i)
=p(\ii;\, \tilde\mm,i),
\end{equation*}
where $\tilde\nn$ denotes the projection of $\nn$ onto $(A_{1},,\ldots,A_{K})$.
This is the consistency with respect to merging of classes of the multivariate hypergeometric distribution, and so the result now follows by the same argument at the end of the proof of Theorem \ref{thm: FV propagation}.
%
\qed

\bigskip
We conclude by proving the recursive representation of Proposition \ref{prop: algorithm FV}, whose argument is analogous to the FV case.

\bigskip
\noindent 
\textbf{Proof of Proposition \ref{prop: algorithm DW}}\\
The update operation \eqref{DW algorithm update} follows directly from Lemma \ref{lem: FV update multi}.
The prediction operation \eqref{eq: FV algorithm propagation} for elements of $\F_{\Pi}$ follows from Theorem \ref{thm: DW propagation} together with the linearity of \eqref{prediction operator} and a rearrangement of the sums, so that
\begin{align*}
\psi_{t}\Bigg(&\,\sum_{\mm\in M}w_{\mm}\Gamma^{\beta+\ss}_{\alpha+\sum_{i=1}^{K_{m}}m_{i}\delta_{y_{i}^{*}}}\Bigg)\\
=&\,\sum_{\mm\in M}w_{\mm}
\sum_{\substack{\nn\in L(\mm)} }
p_{\mm,\nn }(t)
\Gamma^{\beta+\SS_{t}}_{\alpha+\sum_{i=1}^{K_{m}}m_{i}\delta_{y_{i}^{*}}}\\
=&\,\sum_{\nn\in L(M)}
\Bigg(\sum_{\mm\in M,\,\mm\ge \nn}w_{\mm}p_{\mm,\nn}(t)\Bigg)
\Gamma^{\beta+\SS_{t}}_{\alpha+\sum_{i=1}^{K_{m}}m_{i}\delta_{y_{i}^{*}}}.
\end{align*}\qed

As a final comment  concerning the strategy followed for proving the propagation result in Theorems \ref{thm: FV propagation} and \ref{thm: DW propagation}, one could be tempted to work directly
with the duals of the FV and DW processes  \citep{DH82,EK93,E00}. However, this is not optimal, due to the high degree of generality of such dual processes. The simplest path for deriving the propagation step for the nonparametric signals appears to be resorting to the corresponding parametric dual by means of projections and by exploiting the filtering results for those cases.

\bigskip

\end{document}